\documentclass[12pt]{amsart}
\usepackage{amscd,amssymb,amsthm,verbatim}
\newcommand{\Aut}{\operatorname{Aut}}

\newcommand{\C}{{\mathbb C}}

\newcommand{\Diff}{\operatorname{Diff}}

\newcommand{\GL}{\operatorname{GL}}

\newcommand{\Hess}{\operatorname{Hess}}

\newcommand{\Id}{\operatorname{Id}}

\newcommand{\Image}{\operatorname{Im}}

\newcommand{\inv}{\operatorname{inv}}

\newcommand{\Isom}{\operatorname{Isom}}

\newcommand{\Ker}{\operatorname{Ker}}

\newcommand{\pt}{\operatorname{pt}}

\newcommand{\R}{{\mathbb R}}

\newcommand{\Riem}{\operatorname{Riem}}

\newcommand{\SU}{\operatorname {SU}}
\newcommand{\supp}{\operatorname{supp}}

\newcommand{\vol}{\operatorname{vol}}
\newcommand{\Z}{{\mathbb Z}}

\numberwithin{equation}{section}
\theoremstyle{plain}
\newtheorem{definition}{Definition}

\newtheorem{lemma}{Lemma}
\newtheorem{theorem}{Theorem}
\newtheorem{proposition}{Proposition}
\newtheorem{corollary}{Corollary}

\theoremstyle{remark}

\newtheorem{remark}{Remark}
\newtheorem{example}{Example}

\begin{document}
\title{Collapsing with a lower bound on the curvature operator}
\author{John Lott}
\address{Department of Mathematics\\
University of California - Berkeley\\
Berkeley, CA  94720-3840\\ 
USA} \email{lott@math.berkeley.edu}

\thanks{This research was partially supported
by NSF grant DMS-1207654}
\date{February 7, 2014}

\begin{abstract}
Cheeger and Gromov showed that F-structures are related to collapse with a double-sided curvature bound.
We define fibered F-structures and extend some of the Cheeger-Gromov results to the setting of collapse with
a lower bound on the curvature operator.
\end{abstract}

\maketitle


\section{Introduction} \label{section1}

A sequence  $\{g_i\}_{i\: = \:1}^\infty$ of Riemannian metrics on a compact manifold $M$ makes it {\em collapse} to a
lower-dimensional metric space $Y$  if 
the metric spaces $\{(M, g_i)\}_{i\: = \:1}^\infty$ converge to $Y$ in the Gromov-Hausdorff topology.  
If one imposes some curvature condition on the Riemannian metrics then the collapsing can force
extra structure on $M$.  The best understood case is when the Riemannian metrics have a uniform
double-sided curvature bound.  In that case, Cheeger and Gromov showed that collapsing occurs if and only if
$M$ admits an {\em F-structure} of positive rank \cite{Cheeger-Gromov (1986),Cheeger-Gromov (1990)}.
In later work, Cheeger-Fukaya-Gromov gave a complete description of the collapsing geometry in terms of Nil-structures
\cite{Cheeger-Fukaya-Gromov (1992)}.

With the success of the Cheeger-Gromov work, attention turned to collapsing under other curvature bounds.
The case of a uniform lower bound on {\em sectional curvature} was considered by Fukaya-Yamaguchi,
Kapovitch-Petrunin-Tuschmann and others; see \cite{Kapovitch-Petrunin-Tuschmann (2010)} and references therein.
The case of a uniform lower bound on {\em Ricci curvature} was considered by Gromov, Cheeger-Colding,
Kapovitch-Wilking and others; see \cite{Kapovitch-Wilking (2011)} and references therein. By intricate arguments,
it was shown in these cases that the collapsing manifold has a weak fibration structure.

The {\em curvature operator} is the symmetric operator on $2$-forms coming from the Riemann curvature tensor.
The results in this paper indicate that if one assumes a uniform lower bound on the
curvature operator then the collapsing picture is structurally similar to the bounded curvature case considered
by Cheeger and Gromov.  The difference between collapse with curvature operator bounded below, and
collapse with sectional curvature bounded below, is closely tied to the difference between actions
of abelian Lie groups and actions of nonabelian Lie groups; see Section \ref{section2} for examples.

One advantage of the curvature operator is that one knows the manifolds that admit a Riemannian metric
with nonnegative curvature operator, from work of  B\"ohm-Wilking \cite{Boehm-Wilking (2008)}.
This is contrast to the situation with nonnegative sectional curvature.

The Cheeger-Gromov notion of an F-structure is based on local torus actions.  If $M$ has an F-structure
then for each $m \in M$, there is a neighborhood of $m$ with a finite cover that admits a torus action.
The actions by these tori, whose dimensions can vary, have to be consistent on overlaps. 
In this paper we define
a {\em fibered F-structure}, or an fF-structure.  This structure mixes local torus actions and local
fiberings. The motivation for fF-structures comes from collapsing examples, such as those in Section \ref{section2}.

\begin{definition} \label{def1}
An atlas for an fF-structure on a smooth connected manifold $M$ is given by a locally finite open cover $M \: = \: \bigcup_i U_i$
and for each $i$,
\begin{enumerate}
\item A finite normal cover
$\pi_i : \widehat{U}_i \rightarrow U_i$, say with covering group $\Gamma_i$,
\item An effective action $\rho_i : T_i \rightarrow \Diff(\widehat{U}_i)$ of a torus group $T_i$, which is
$\Gamma_i$-equivariant with respect to a homomorphism $\Gamma_i \rightarrow \Aut(T_{i})$,  and
\item A  $T_i \rtimes \Gamma_i$-equivariant  proper submersion $s_i : \widehat{U}_i \rightarrow
\widehat{B}_i$ so that
\item The $T_i$-action is never vertical; given $X$ in the Lie algebra 
${\mathfrak t}_i$ and $\widehat{p} \in \widehat{U}_{i}$,
if the corresponding vector field $V_X$ has $V_X(\widehat{p}) \in \Ker(d{s_i})$ then
$V_X(\widehat{p}) \: = \: 0$. 
\end{enumerate}
\end{definition}

There is a consistency requirement on overlaps, which is detailed in Section \ref{section3}. 

One can extend notions from F-structures to fF-structures.
An fF-structure has {\em positive rank} if for each $i$,
$\dim(\widehat{B}_i) < \dim(\widehat{U}_i)$ or 
the $T_i$-action has no global fixed points. 
If for each $i$, the orbits of the $T_i$-action on $\widehat{U}_i$ all have the same dimension
(possibly depending on $i$),
then the fF-structure is {\em polarized}. If $\dim(T_i)$ is independent of $i$, and
$\dim(s_i^{-1}(\pt))$ is independent of $i$, then the
fF-structure is {\em pure}.  More generally,
there is a notion of a substructure, whose actions are by
subgroups (not necessarily closed) of the toral groups, so it makes sense to talk about
an fF-structure with  a polarized substructure.

There is a notion of an {\em invariant vertical metric} for the fF-structure.  Invariant vertical metrics always exist.
There is also a notion of an {\em invariant metric}.
Given an invariant vertical metric, one can extend it to an invariant metric.

The Cheeger-Gromov results about bounded curvature collapse are phrased in terms of the injectivity radius. Although the injectivity
radius is not suitable for other types of collapse, one can instead use local volumes; with bounded curvature,
the two notions of collapse are equivalent \cite{Cheeger (1970),Cheeger-Gromov-Taylor (1982)}.
Given $r > 0$, $m \in M$ and a Riemannian metric $g$ on $M$, let $V_{r,m}(g)$ denote the volume of the ball of radius $r$ around $m$.
The next three theorems extend the main results of \cite{Cheeger-Gromov (1986)} to the setting of collapse with a lower
bound on the curvature operator.

\begin{theorem} \label{thm1}
Suppose that $M$ is a compact manifold with a pure polarized fF-structure of positive rank. 
Suppose that there is an invariant vertical metric with nonnegative vertical curvature operator. 
Then there is a
family $\{g_\epsilon\}_{\epsilon > 0}$ of invariant metrics on $M$ so that 
\begin{enumerate}
\item The curvature operators of $\{(M, g_{\epsilon})\}_{\epsilon > 0}$ are uniformly bounded below.
\item The diameters of $\{(M, g_{\epsilon})\}_{\epsilon > 0}$ are uniformly bounded above.
\item The Gromov-Hausdorff limit $\lim_{\epsilon \rightarrow 0} (M, g_\epsilon)$ is the
orbit space of the fF-structure, a lower-dimensional length space.
\end{enumerate}
\end{theorem}

\begin{theorem} \label{thm2}
Suppose that $M$ is a compact manifold with a fF-structure, having a polarized substructure of positive rank.  Suppose
that there is an invariant metric on $M$ for which the substructure has a positive vertical curvature operator.  Then there is a
family $\{g_\epsilon\}_{\epsilon > 0}$ of invariant metrics on $M$ so that 
\begin{enumerate}
\item The curvature operators of $\{(M, g_{\epsilon})\}_{\epsilon > 0}$ are uniformly bounded below.
\item $\lim_{\epsilon \rightarrow 0} \vol(M, g_\epsilon) \: = \: 0$.
\end{enumerate}
\end{theorem}

\begin{theorem} \label{thm3}
Suppose that $M$ is a compact manifold with an fF-structure of positive rank.  Suppose
that there is an invariant vertical metric with nonnegative vertical curvature operator.  Then there is a
family $\{g_\epsilon\}_{\epsilon > 0}$ of invariant metrics on $M$ so that 
\begin{enumerate}
\item The curvature operators of $\{(M, g_{\epsilon})\}_{\epsilon > 0}$ are uniformly bounded below.
\item  For any $r > 0$, 
we have $\lim_{\epsilon \rightarrow 0} \sup_{m \in M} V_{r, m}(g_\epsilon) \: = \: 0$.
\end{enumerate}
\end{theorem}

\begin{corollary} \label{cor1}
Under the assumptions of Theorem \ref{thm3}, the simplicial volume of $M$ vanishes.
\end{corollary}
\begin{proof}
This follows from Theorem \ref{thm3} and the isolation theorem of \cite[p. 14]{Gromov (1982)}.
\end{proof}

In the proofs of Theorems \ref{thm1}-\ref{thm3}, in each case there is a potentially dangerous
term in the curvature operator that must be handled.  The way to handle this term
differs in the three cases. For  Theorem \ref{thm1}, we use the Ricci flow to
reduce the structure group of the fiber bundle. For  Theorem \ref{thm2}, we
use the positivity of the vertical curvature operator. For Theorem \ref{thm3},
we expand in transverse directions.

The converses of the F-structure versions of Theorems \ref{thm1} and \ref{thm2} fail
\cite[Example 6.4]{Cheeger-Rong (1995)},\cite[Remark 4.1]{Cheeger-Gromov (1986)}.
The converse of the F-structure version of Theorem \ref{thm3} is the main result of
\cite{Cheeger-Gromov (1990)}. In Section \ref{section8} we give a local model for the
geometry of a manifold that satisfies the conclusion of Theorem \ref{thm3}, under a smoothing 
assumption.  We show that the
local model has a canonical fF-structure with nonnegative vertical curvature operator.  It seems plausible that if $M$ satisfies
the conclusion of Theorem \ref{thm3} then after removing a finite number of disjoint topological balls,
the remainder has an fF-structure of positive rank, with fibers that admit metrics having
nonnegative curvature operator; see Remark \ref{remark2} for further discussion. 

My interest in these questions came from a talk by Wilderich Tuschmann about his work with
Herrmann and Sebastian on manifolds with almost nonnegative curvature operator, which was
written up in \cite{Herrmann-Sebastian-Tuschmann (2012)}.  I thank Wilderich and
Peter Teichner for discussions.

\section{Nonnegative curvature operator} \label{section2}

In this section we give information about manifolds with nonnegative curvature operator.
We also give some examples of collapsing.

If $M$ is a Riemannian manifold then the curvature operator at $p \in M$ is the self-adjoint map $\Riem$
on $\Lambda^2(T^*_pM)$ given in local coordinates by
\begin{equation} \label{2.1}
\Riem(\omega) \: = \: \sum_{i,j,k,l} R_{ij}^{\: \: \: \: kl} \omega_{kl} \:
dx^i \wedge dx^j.\
\end{equation}
We can also consider $\Riem$ as an operator on $\Lambda^2(T^*_pM) \otimes \C$.
Clearly $\Riem \ge 0$ implies that $M$ has nonnegative sectional curvature but the converse
fails in dimension greater than three.

Any  symmetric space $M \: = \: G/H$ with nonnegative sectional curvature has nonnegative
curvature operator, as follows from the equation $R(X,Y)Z \: = \: [Z, [X, Y]]$ for
$X,Y,Z \in T_{eH} M$. A K\"ahler manifold of real dimension greater than two cannot have
positive curvature operator, since $\Riem$ vanishes on $\Lambda^{2,0}(T^*_pM) \oplus
\Lambda^{0,2}(T^*_pM)$. It will have nonnegative curvature operator if and only if
$\Riem$ is nonnegative on $\Lambda^{1,1}(T^*_pM)$.

B\"ohm and Wilking proved that a compact connected Riemannian manifold with
positive curvature operator is diffeomorphic to a spherical space form
\cite{Boehm-Wilking (2008)}. When combined with \cite[Theorem 5]{Chow-Yang (1989)},
we can say that if
$M$ is a compact connected Riemannian manifold with nonnegative curvature operator
then each factor in the de Rham decomposition of the universal cover $\widetilde{M}$ is isometric to 
one of
\begin{enumerate}
\item A Euclidean space.
\item A sphere with nonnegative curvature operator.
\item A  compact irreducible symmetric space.
\item A compact K\"ahler manifold that is biholomorphic to a complex projective space, with $\Riem$ nonnegative on real $(1,1)$-forms.
\end{enumerate}

Let $M$ be a complete connected Riemannian manifold with nonnegative curvature operator. 
Since $M$ also has nonnegative sectional curvature, it is diffeomorphic to a vector bundle
over its soul $S$, in the sense of Cheeger-Gromoll \cite{Cheeger-Gromoll (1972)}. 
The nonnegativity of the curvature operator implies
that there is a local isometric product structure over the soul \cite{Noronha (1989)}. That is, $M$ is
an isometric quotient
$(\widetilde{S} \times Y)/\Gamma$, where
\begin{enumerate}
\item $S$ is a compact Riemannian manifold having nonnegative curvature operator, with fundamental group $\Gamma$ and
 $\widetilde{S}$ as its universal cover, 
\item $(Y, y)$ is a pointed complete Riemannian manifold with nonnegative curvature operator that is diffeomorphic to $\R^k$, with $k \in [0,n]$, and
\item $\Gamma$ acts by isometries on $Y$, fixing $y$.
\end{enumerate}
That $\Gamma$ fixes a point $y \in Y$ comes from the fact that
the soul $S$ is a submanifold of $M$, and so must be
representable as $(\widetilde{S} \times \{y\})/\Gamma$.
We remark that the local isometric product structure also exists under the weaker assumption that $M$ has nonnegative
complex sectional curvature \cite[Theorem 2]{Cabezas-Rivas-Wilking (2011)}.

\begin{example} \label{example1}
Although $S^{2n+1}$ has positive curvature operator, $\C P^n \: = \: S^{2n+1}/S^1$ does not.  This shows that a
lower bound on the curvature operator is not preserved under taking quotients.
\end{example}

\begin{example} \label{example2}
Although $\SU(3)$, with a bi-invariant Riemannian metric, has nonnegative curvature operator, its quotient $\SU(3)/S^1$ by a circle subgroup cannot.
Rescaling, we see that a manifold with a nonnegative curvature operator can have a quotient with
arbitrarily negative curvature operator.
\end{example}

\begin{example} \label{example3}
Suppose that $T^l$ acts isometrically on a compact Riemannian manifold $M$ with curvature operator
bounded below by $K \Id$. Let $\Z_k \subset S^1$ be
the finite subgroup of order $k$.
For $\epsilon > 0$, let $\epsilon T^l$ denote the result of taking a flat $T^l$ and multiplying the Riemannian
metric by $\epsilon^2$. Put
$Y_{k, \epsilon} \: = \: M \times_{\Z_k^l} \epsilon T^l$. The curvature operator on $Y_{k,\epsilon}$ is
bounded below by $K \Id$. As $k \rightarrow \infty$ and $\epsilon \rightarrow 0$, the
Gromov-Hausdorff limit of $Y_{k,\epsilon}$ is $M/T^l$. Thus $M/T^l$ is a limit of Riemannian manifolds with 
curvature operators that are uniformly bounded from below.
\end{example}

\begin{example} \label{example4}
Suppose that a compact Lie group $G$ acts isometrically on a compact Riemannian manifold $M$
with nonnegative sectional curvature. Give $G$ a bi-invariant Riemannian metric.  For $\epsilon > 0$,
put $M_\epsilon \: = \: M \times_G \epsilon G$, equpped with the quotient Riemannian metric coming from $M \times \epsilon G$.
As $\epsilon \rightarrow 0$, the Gromov-Hausdorff limit of
$M_\epsilon$ is $M/G$.
For each $\epsilon$, the manifold $M_\epsilon$ has nonnegative sectional curvature. 
 Thus $M/G$ is a limit of Riemannian manifolds with nonnegative sectional curvature.
In contrast, Proposition \ref{prop2} in
Section \ref{section8} shows that if $G$ is nonabelian then
it is generally not true that the family  $\{M_\epsilon\}_{\epsilon \in (0,1]}$
has curvature operators that are uniformly bounded from below.
\end{example}

\begin{example} \label{example5}
Let $M$ be a compact Riemannian manifold. Let $G$ be a compact Lie group and 
let $P$ be a  principal $G$-bundle over $M$, with connection. Give $G$ a bi-invariant
Riemannian metric and give $P$ the associated connection metric $g$. For $\epsilon > 0$, let $(P, g_\epsilon)$ be the
result of multiplying the fiber Riemannian metric by $\epsilon^2$. The following facts can be read off
from (\ref{4.7}) below.
\begin{enumerate}
\item  If $M$ has nonnegative sectional curvature then $P$ has almost nonnegative sectional curvature :
for every $\sigma > 0$,
the manifold $(P, g_\epsilon)$ has sectional curvatures bounded below by $- \sigma$ if
$\epsilon$ is sufficiently small \cite[Section 2]{Fukaya-Yamaguchi (1992)}.
\item If $M$ has nonnegative curvature operator and $G$ is abelian then $P$ has almost nonnegative curvature
operator :
for every $\sigma > 0$,
the manifold $(P, g_\epsilon)$ has curvature operator bounded below by $- \sigma \Id$ if
$\epsilon$ is sufficiently small. 
\item If $M$ has nonnegative curvature operator and $G$ is nonabelian then
the manifolds $\{(P, g_\epsilon)\}_{\epsilon \in (0,1]}$ have curvature operators that are uniformly bounded below. 
However, it is generally not true that
for every $\sigma > 0$,
the manifold $(P, g_\epsilon)$ has curvature operator bounded below by $- \sigma \Id$ if
$\epsilon$ is sufficiently small.
\end{enumerate}
\end{example}

\begin{remark}
As seen, if $s : M \rightarrow B$ is a Riemannian submersion and $M$ has nonnegative curvature operator then $B$ need not have
nonnegative curvature operator.  Hence there cannot be a synthetic notion of nonnegative curvature operator which is
preserved under Gromov-Hausdorff limits, at least in the collapsing case.  This is in contrast to the situation with 
nonnegative sectional curvature.  It may be useful to think of collapsing with curvature operator bounded below as a
type of collapsing with sectional curvature bounded below, with some additional structure. 
\end{remark}

\section{fF-structures} \label{section3}

In this section we define fF-structures and give some examples.  The definition is based, of course, on the Cheeger-Gromov definition of F-structures.
We show that an invariant vertical metric always exists, and that an
invariant vertical metric can always be extended to an invariant metric.

We first give a definition of infinitesimal fF-structures in terms of sheaves.  Some readers may want to 
skip to Definition \ref{def4}  of an atlas.

Let $M$ be a smooth connected manifold. Let ${\mathcal V}$ denote the sheaf on $M$ of smooth vector fields, a sheaf of Lie algebras.
For $m \in M$, let ${\mathcal V}_m$ denote the stalk at $m$, i.e. the germs of vector fields at $m$, and let
$l_m : {\mathcal V}_m  \rightarrow T_mM$ be the localization map, i.e. evaluation at $m$.

\begin{definition} \label{def2}
An infinitesimal fF-structure on $M$ consists of
\begin{enumerate}
\item A sheaf ${\mathcal K}$ of finite-dimensional abelian Lie algebras on $M$,
\item An injective sheaf homomorphism $\eta : {\mathcal K} \rightarrow {\mathcal V}$ and
\item A subsheaf ${\mathcal D}$ of ${\mathcal V}$ so that for any open set $U \subset M$,
 ${\mathcal D}(U)$ consists of the sections of an integrable distribution $D_U \subset TU$,
\end{enumerate}
satisfying the conditions that
\begin{enumerate}
\item If $V_1 \in {\mathcal K}(U)$ and $V_2 \in {\mathcal D}(U)$ then $[\eta(V_1), V_2] \in {\mathcal D}(U)$, and
\item For every $m \in M$, $l_m(\eta_m({\mathcal K}_m)) \cap l_m({\mathcal D}_m)$ vanishes in $ T_mM$.
\end{enumerate}
\end{definition}

The first condition in Definition \ref{def2} is an equivariance statement about ${\mathcal D}$.
The second condition in Definition \ref{def2} says that a tangent vector at $m$ coming from the
infinitesimal Lie algebra action can never point in the direction of the distribution.

Given an infinitesimal fF-structure, a subinfinitesimal fF-structure is given by a subsheaf 
$\left( {\mathcal K}^\prime, {\mathcal D}^\prime \right)$ with
similar properties.

The infinitesimal fF-structure has {\em positive rank} if for all $m \in M$,
the pair $( l_m(\eta_m({\mathcal K}_m)), l_m({\mathcal D}_m)) \subset T_mM \oplus T_mM$ never vanishes. 
An infinitesimal fF-structure is {\em pure}
if the dimensions of ${\mathcal K}_m$ and $l_m({\mathcal D}_m)$ are constant in $m$.

\begin{definition} \label{def3}
An fF-structure on $M$ is an infinitesimal fF-structure with the property that for each $m \in M$ there
are a neighborhood  $U_m$ of $m$ and a finite normal cover $\pi_m : \widehat{U}_m \rightarrow U_m$, say with
covering group $\Gamma_m$, so that
\begin{enumerate}
\item There is an effective action $\rho : T_m \rightarrow \Diff(\widehat{U}_m)$ of a torus group $T_m$, equivariant with respect to a
homomorphism $\Gamma_m \rightarrow \Aut(T_m)$, and
\item There is a proper submersion $s_m : \widehat{U}_m \rightarrow \widehat{B}_m$, which is
$T_m \rtimes \Gamma_m$-equivariant, such that
\item If $\widehat{m} \in \pi_m^{-1}(m)$ then the action near $\widehat{m}$ of the Lie algebra ${\mathfrak t}_m$
coincides with the lift of $\eta_m({\mathcal K}_m)$, and
\item If $\widehat{m} \in \pi_m^{-1}(m)$ then the germs (at $\widehat{m}$)  of sections  of the vertical tangent bundle $\Ker(ds_m)$
coincide with the lift of ${\mathcal D}_m$.
\end{enumerate}
\end{definition}
In conditions (3) and (4) above, we use the fact that $\pi_m$ is a local diffeomorphism in order to lift vectors.

We say that the fF-structure has {\em positive rank} if the underlying infinitesimal fF-structure has positive rank.
We say that the fF-structure is {\em pure} if the underlying infinitesimal fF-structure is pure.

\begin{definition} \label{def4}
An atlas for an $fF$-structure on $M$ is given by a locally finite open cover $M \: = \: \bigcup_i U_i$
and for each $i$,
\begin{enumerate}
\item A finite normal cover
$\pi_i : \widehat{U}_i \rightarrow U_i$, say with covering group $\Gamma_i$,
\item An effective action $\rho_i : T_i \rightarrow \Diff(\widehat{U}_i)$ of a torus group $T_i$, which is
$\Gamma_i$-equivariant with respect to a homomorphism $\Gamma_i \rightarrow \Aut(T_{i})$  and
\item A  $T_i \rtimes \Gamma_i$-equivariant  proper submersion $s_i : \widehat{U}_i \rightarrow
\widehat{B}_i$ so that
\item The $T_i$-action is never vertical; given $X$ in the Lie algebra 
${\mathfrak t}_i$ and $\widehat{p} \in \widehat{U}_{i}$,
if the corresponding vector field $V_X$ has $V_X(\widehat{p}) \in \Ker(d{s_i})$ then
$V_X(\widehat{p}) \: = \: 0$. 
\end{enumerate}

This structure must have the following intersection property.  Suppose that $U_i \cap U_j \neq \emptyset$.
Put $U_{ij} \: = \: U_i \cap U_j$. Then there are
\begin{enumerate}
\item A finite normal $\Gamma_{ij}$-cover
$\pi_{ij} : \widehat{U}_{ij} \rightarrow U_{ij}$,
\item An effective action $\rho_{ij} : T_{ij} \rightarrow \Diff(\widehat{U}_{ij})$ of a torus group $T_{ij}$,
which is $\Gamma_{ij}$-equivariant
with respect to a homomorphism $\Gamma_{ij} \rightarrow \Aut(T_{ij})$,
\item  A  $T_{ij} \rtimes \Gamma_{ij}$-equivariant  proper submersion $s_{ij} : \widehat{U}_{ij} \rightarrow
\widehat{B}_{ij}$ so that the $T_{ij}$-action is never vertical,
\item A subgroup $T_{ij,i} \rtimes \Gamma_{ij,i} \subset T_{ij} \rtimes\Gamma_{ij} $ with a locally isomorphic surjective homomorphism
$T_{ij,i} \rtimes \Gamma_{ij,i} \rightarrow T_i \rtimes \Gamma_i$,
\item  A subgroup $T_{ij,j} \rtimes \Gamma_{ij,j} \subset T_{ij} \rtimes\Gamma_{ij} $ with a locally isomorphic surjective homomorphism
$T_{ij,j} \rtimes \Gamma_{ij,j} \rightarrow T_j \rtimes \Gamma_j$  and
\item Equivariant smooth maps 
$\alpha_{ij,i} : \widehat{U}_{ij} \rightarrow \widehat{U}_i$,
$\alpha_{ij,j} : \widehat{U}_{ij} \rightarrow \widehat{U}_j$,
$\beta_{ij,i} : \widehat{B}_{ij} \rightarrow \widehat{B}_i$,
$\beta_{ij,j} : \widehat{B}_{ij} \rightarrow \widehat{B}_j$
\end{enumerate}
so that
\begin{enumerate}
\item
$\pi_i \circ \alpha_{ij,i} \: = \: \pi_{ij}$,
\item
$\pi_j \circ \alpha_{ij,j} \: = \: \pi_{ij}$,
\item
$s_i \circ \alpha_{ij,i} \: = \: \beta_{ij,i} \circ s_{ij}$ and
\item
$s_j \circ \alpha_{ij,i} \: = \: \beta_{ij,j} \circ s_{ij}$.
\end{enumerate}
\end{definition}

In condition (6) above the equivariance of $\alpha_{ij,i}$, for example, is with respect to the
$T_{ij,i} \rtimes \Gamma_{ij,i}$-action on $\widehat{U}_{ij}$ and the
$T_{i} \rtimes \Gamma_{i}$-action on $\widehat{U}_i$, as linked by the
homomorphism in (4).

Any fF-structure admits a compatible atlas in the sense of Definition \ref{def4}; cf. \cite[p. 317]{Cheeger-Gromov (1986)}.  We can and will
assume that for each $m \in M$, we have $T_m \: = \: T_i$ for some $i$. 
The fF-structure has positive rank if and only if for all $i$, the $T_i$-action has no global fixed points
or the preimages $s_i^{-1}(\pt)$ have positive dimension. The fF-structure  is pure if and only if
$\dim(T_i)$ is independent of $i$ and $\dim(s_i^{-1}(\pt))$ is independent of $i$.

Given $m \in M$, define the vertical fiber $F_i(m) \subset M$ to be
the image under $\pi_i$ of the vertical fibers in $\widehat{U}_i$ that contain elements of $\pi_i^{-1}(m)$. That is, 
$F_i(m) \: = \: \pi_i ( s_i^{-1} (s_i ( \pi_i^{-1}(m))))$.

Given an fF-structure on $M$, suppose that the underlying infinitesimal fF-structure has a subinfinitesimal fF-structure $({\mathcal K}^\prime, {\mathcal D}^\prime)$.
We say that  $({\mathcal K}^\prime, {\mathcal D}^\prime)$ is {\em polarized} if the fF-structure admits
an atlas so that for each $i$, if we partition $U_i$ by the equivalence relation generated by flows of vector fields in $\eta({\mathcal K}^\prime(U_i))$,
then the ensuing orbits have constant dimension in $U_i$, as immersed submanifolds.

We do not assume that the lift of
$\eta({\mathcal K}^\prime(U_i))$ to $\widehat{U}_i$ generates a {\em closed} subgroup of the $T_{i}$-action on 
$\widehat{U}_i$. 
We also do not assume that the flows of ${\mathcal D}^\prime(U_i)$ generate {\em closed} orbits, or that 
$\dim( {\mathcal K}^\prime(U_i) )$ or $\dim(D^\prime_{U_i})$ are  independent of $i$.  If in addition
$({\mathcal K}^\prime, {\mathcal D}^\prime)$ is pure, i.e. if
$\dim( {\mathcal K}^\prime(U_i) )$ and $\dim(D^\prime_{U_i})$ are both independent of $i$, then
$({\mathcal K}^\prime, {\mathcal D}^\prime)$ is {\em pure polarized}. Our terminology differs slightly from that
for F-structures in \cite{Cheeger-Gromov (1986)}, where
one would say that the existence of such a substructure means that the {\em original} F-structure has a polarization, or a pure polarization.

\begin{example}
Let $M$ have an F-structure.  Then for any connected closed manifold $Z$, the product manifold $M \times Z$ has an fF-structure. 
It will have
positive rank if and only if $Z$ has positive dimension or the F-structure on $M$ has positive rank. 
\end{example}

\begin{example}
Let $Z_1$ and $Z_2$ be connected closed manifolds. Let $M_1$ and $M_2$ be connected manifolds. 
Let $B$ be a manifold.  Suppose that there are codimension-zero embeddings $(Z_1 \times B) \subset M_1$ and
$(Z_2 \times B) \subset M_2$. Then $M = (Z_2 \times M_1) \cup_{Z_1 \times Z_2 \times B} (Z_1 \times M_2)$ has an fF-structure.
There is an evident atlas whose groups $T_i$ are trivial and whose submersions are $(Z_2 \times M_1) \rightarrow M_1$,
$(Z_1 \times M_2) \rightarrow M_2$ and $(Z_1 \times Z_2 \times B) \rightarrow B$.
\end{example}

\begin{example}
Let $G$ be a connected 
compact Lie group. Let $K$ be a closed subgroup of $G$. Let $P$ be the total space of a principal $G \times T^k$-bundle.
Then $P/K$ has a pure polarized  fF-structure coming from the submersion $s : P/K \rightarrow P/G$, along with the remaining $T^k$-action.
 
\end{example}
 
\begin{lemma} \label{lem1}
Given an fF-structure, there is a compatible atlas with the following properties.
\begin{enumerate}
\item Each $U_i$ has compact closure.
\item If $m \in U_{i_1} \cap \ldots \cap U_{i_k}$ then for some ordering, we have
\begin{equation} \label{3.1}
(T_{i_1}, F_{i_1}(m)) \subset (T_{i_2}, F_{i_2}(m)) \subset \ldots \subset (T_{i_k}, F_{i_k}(m)).
\end{equation}
\item Given $m \in U_i$, there is at most one $j \neq i$ so that $m \in U_j$ and
$(T_{i}, F_{i}(m)) \: = \: (T_{j}, F_{j}(m))$.
\end{enumerate}
\end{lemma}
\begin{proof}
The proof is the same as in \cite[Lemma 1.2]{Cheeger-Gromov (1986)}.
\end{proof}

The relation of  (\ref{3.1}) gives a partial ordering
on $\{U_\alpha\}$, in the sense that $U_\alpha \le U_\beta$ if for each $m \in U_\alpha \cap U_\beta$, we have
$(T_\alpha, F_\alpha(m)) \subset (T_\beta, F_\beta(m))$. Let $U_1$ be a maximal element of the partial ordering.  Let
$U_2$ be a maximal element among those that are left after removing $U_1$. Proceeding in this way,  we obtain a 
subcover $\{U_i\}$ of  $M$ with the property 
that if $i > j$ and $U_i \cap U_j \neq \emptyset$ then the restriction of $(T_i, F_i)$, to $U_i \cap U_j$, is contained in the restriction of
$(T_j, F_j)$ to $U_i \cap U_j$. We will call such an atlas a {\em regular atlas}.

The flows generated by the vector fields in the ${\mathcal D}(U)$'s generate a partition ${\mathcal P}$ of $M$ into
submanifolds, possibly of varying dimensions. If $m \in U_i$ then the vertical fiber $F_i(m)$ is contained in the submanifold 
$P_m \in {\mathcal P}$ that contains $m$. In fact, $P_m = \bigcup_{i : m \in U_i} F_i(m)$.

Given a collection $g^V$ of Riemannian metrics on the
submanifolds $P$ in ${\mathcal P}$, for any  atlas we obtain vertical Riemannian metrics on the submersions
$s_i : \widehat{U}_i \rightarrow \widehat{B}_i$. 

\begin{definition} \label{def5}
We say that $g^V$ is an {\em invariant vertical metric} on $M$
if for each $i$, the vertical Riemannian metric on $\widehat{U_i}$ is smooth and $T_i$-invariant. 
We say that $g^V$ has positive (nonnegative) curvature operator if for each $i$, the
vertical Riemannian metric on $\widehat{U_i}$ has positive (nonnegative) curvature operator.
\end{definition}

This notion is independent of the
choice of atlas.

\begin{lemma} \label{lem2}
An invariant vertical metric exists.
\end{lemma}
\begin{proof}
The proof is similar to that of \cite[Lemma 1.3]{Cheeger-Gromov (1986)}. For simplicity, we assume that $M$ is compact.
The proof can be easily modified to the noncompact case.

Let $\{ U_i \}$ be a regular atlas. After refining the atlas if necessary, we can assume that for each $i$,
the closure $\overline{U_i}$ has a normal $\Gamma_i$-cover $\widehat{\overline{U_i}}$ which fibers over a smooth
compact manifold-with-boundary $\widehat{\overline{B_i}}$.

Choose a vertical Riemannian metric on $\widehat{U}_{1}$ which extends smoothly to $\widehat{\overline{U_1}}$.
Average it with respect to the
$T_1$-action. Project to get an invariant vertical  metric on 
$U_{1}$.

Consider $U_2$. The invariant vertical metric on $U_1 \cap U_2$ 
(possibly empty) lifts to a vertical Riemannian metric on $\pi_2^{-1}(U_1 \cap U_2)$. Extend this to
a vertical Riemannian metric on $\widehat{U}_2$ which extends smoothly to $\widehat{\overline{U_2}}$.
Average it with respect to the $T_2$-action. This new averaging will not change the
vertical Riemannian metric on $\pi_2^{-1}(U_1 \cap U_2)$. Project to get an invariant vertical 
metric on $U_1 \cup U_2$ 
which agrees with the previous invariant vertical metric on $U_1 \cap U_2$.
Then repeat the process.
\end{proof}

\begin{definition} \label{def6}
Given an infinitesimal fF-structure, a smooth function $f$ on $M$ is {\em invariant} if for
all open $U \subset M$ and all $V \in \eta({\mathcal K}(U)) + {\mathcal D}(U)$, we have
$V f \Big|_U \: = \: 0$.

A Riemannian metric $g$ on $M$ is an {\em invariant metric} if for all open $U \subset M$,
\begin{enumerate}
\item If $V \in \eta({\mathcal K}(U))$ then ${\mathcal L}_V g \: = \: 0$ and $V$ is 
pointwise orthogonal to each element of ${\mathcal D}(U)$.
\item Given $m \in U$ and $e_1, e_2 \in T_mM$, if $e_1, e_2 \in D_U^\perp$ then for
any $V \in {\mathcal D}(U)$, we have $({\mathcal L}_V g)(e_1, e_2) \: = \: 0$.
\end{enumerate}
\end{definition}

Condition (1) in Definition \ref{def6} means that $g$ is invariant under the local torus action, and the generators
are orthogonal to the distribution.  Condition (2) in Definition \ref{def6} gives a local Riemannian submersion
structure, with fibers in the distribution direction.

Given an fF-structure on $M$, we say that $g$ is invariant if it is invariant for the underlying infinitesimal
fF-structure.  If  $\{U_i\}$ is an atlas then an invariant  Riemannian metric $g$ on $M$ restricts to a Riemannian metric on $U_i$, which lifts to
a Riemannian metric on $\widehat{U}_i$. This defines a $T_i \rtimes \Gamma_i$-invariant Riemannian submersion structure on
$s_i : \widehat{U}_i \rightarrow \widehat{B}_i$.

\begin{lemma} \label{lem3}
\begin{enumerate}
\item
Any invariant vertical metric is the restriction of an invariant metric on $M$ to the submanifolds in the
partition ${\mathcal P}$.
\item There are an atlas $\{U_i\}$ and smooth invariant compactly-supported functions $F_i : U_i \rightarrow
[0, 1]$ so that $M \: = \: \bigcup_i F_i^{-1}(1)$.
\end{enumerate}
\end{lemma}
\begin{proof}
For simplicity, we assume that $M$ is compact. 
Let $g^V$ be an invariant vertical metric.  Let $M \: = \: \bigcup_i U_i$ be a regular atlas. We assume that
$\overline{U_i}$ has the properties in the proof of Lemma \ref{lem2}.
The lift of $g^V \big|_{U_1}$ gives a vertical Riemannian metric $\widehat{g}^V_1$ on $\widehat{U}_1$.
As  the generating vector fields of the $T_1$-action are never vertical, one can use the slice theorem (along with an
induction over the strata of the $T_1$-action) to
construct a $\Gamma_i$-invariant horizontal distribution
${\mathcal H}_1$ for $s_1 : \widehat{U}_1 \rightarrow \widehat{B}_1$ 
that extends smoothly to $\widehat{\overline{U_1}}$ and contains the generating
vector fields of the $T_1$-action.
Choose a Riemannian metric $\widehat{g}^H_1$ on $\widehat{B}_1$ that extends smoothly to
$\widehat{\overline{B_1}}$.  There is an induced Riemannian metric $\widehat{g}_1 \: = \: \widehat{g}^V_1 \oplus 
s_1^* \widehat{g}^H_1$ on $\widehat{U}_1$, for which ${\mathcal H}_1$ is orthogonal to $\Ker(ds_1)$.
Average $\widehat{g}_1$ with respect to $T_1 \rtimes \Gamma_1$. The result  pulls back from
an invariant metric $g_1$ on $U_1$.

Consider $U_2$. The lift of $g^V \big|_{U_2}$ gives a vertical Riemannian metric $\widehat{g}^V_2$ on
$\widehat{U}_2$. The lift of $g_1 \big|_{U_1 \cap U_2}$ gives a Riemannian metric $\widehat{g}_1^\prime$ on
$\pi_2^{-1}(U_1 \cap U_2)$ which restricts to $\widehat{g}^V_2$ on the fibers of $s_2$ in $\pi_2^{-1}(U_1 \cap U_2)$.
Let ${\mathcal H}_1^\prime$ be the orthogonal complement to $\Ker(ds_2)$ on $\pi_2^{-1}(U_1 \cap U_2)$,
with respect to $\widehat{g}_1^\prime$.  
As  the generating vector fields of the $T_2$-action are never vertical, one can 
construct a horizontal distribution
${\mathcal H}_2$ for $s_2$ that extends smoothly to $\widehat{\overline{U_2}}$, contains the generating
vector fields of the $T_2$-action and agrees with ${\mathcal H}_1^\prime$ on $\pi_2^{-1}(U_1 \cap U_2)$.

Put $\widehat{B}_1^\prime \: = \: s_2(\pi_2^{-1}(U_1 \cap U_2)) \subset \widehat{B}_2$.
There is a unique Riemannian metric $\widehat{g}^{H \prime}_1$ on
$\widehat{B}_1$ so that $\widehat{g}_1^\prime \: = \: \widehat{g}^V_2 \oplus s_2^* \widehat{g}^{H \prime}_1$
on $\pi_2^{-1}(U_1 \cap U_2)$.
Choose a Riemannian metric $\widehat{g}^H_2$ on $\widehat{B}_2$ that agrees with  $\widehat{g}^{H \prime}_1$ on
$\widehat{B}_1$ and that extends smoothly to
$\widehat{\overline{B_2}}$. 
 There is an induced Riemannian metric $\widehat{g}_2 \: = \: \widehat{g}^V_2 \oplus 
s_2^* \widehat{g}^H_2$ on $\widehat{U}_2$ for which ${\mathcal H}_2$ is orthogonal to $\Ker(ds_2)$.
Average $\widehat{g}_2$ with respect to $T_2 \rtimes \Gamma_2$.  The result  pulls back from
an invariant metric $g_2$ on $U_2$ which agrees with $g_1$ on $U_1 \cap U_2$. Repeat the process.
This produces an invariant metric $g$ on $M$ that restricts to $g^V$.

Fix a bump function $\phi : [0, \infty) \rightarrow [0, \infty)$ so that $\phi \Big|_{[0,1/2]} \: = \: 1$ and
$\phi \Big|_{[1, \infty)} \: = \: 0$. Given $i$,
let $\{P_{i,j}\}$ be a finite collection of  submanifolds in  ${\mathcal P}$ that lie in $U_i$. For $\epsilon_{i,j} >0$ small enough,
the function $A_{i,j}(m) \: = \: \phi \left( d(m, P_{i,j})/\epsilon_{i,j} \right)$ is smooth on $M$, and clearly invariant. 
Put
\begin{equation} \label{3.2}
F_i(m) \: = \: 1 - \phi \left( \sum_j A_{i,j}(m) \right).
\end{equation}
With an appropriate choice of the $P_{i,j}$'s and $\epsilon_{i,j}$'s, the ensuing $F_i$'s will satisfy the conclusion of the lemma.
\end{proof}

\section{Riemannian submersions} \label{section4}

In this section we collect some curvature formulas for Riemannian submersions, after warping both the fibers and the base.

Let $s : M \rightarrow B$ be a Riemannian submersion. Let $g^V$ denote the inner product on $T^VM \: = \: \Ker(ds)$ and let
$g_B$ denote the inner product on $TB$, so $g_M \: = \: g^V \oplus s^* g_B$.  Let $\{e_\alpha\}$ be a local orthonormal frame for the
horizontal space $T^HM$ that pulls back from $B$. Let $\{e_i\}$ be a local orthonormal frame for 
$T^VM$. Put $\{e_I\} \: = \: \{e_\alpha\} \cup \{e_i\}$ and let $\{\tau^I\}$ denote the
dual coframe.  Using the Einstein summation convention, we write
$\omega^I_{\: \: JK} \: = \: \tau^I \left( \nabla_{e_K} e_J \right)$
and $\omega^I_{\: \:J} \: = \: \omega^I_{\: \: JK} \tau^K$
so $\omega^I_{\: \: JK} + \omega^J_{\: \: IK} \: = \: 0$ and $d\tau^I + \omega^I_{\: \: J} \wedge \tau^J \: = \: 0$.
With $\Omega^I_{\: \:J} = d\Omega^I_{\: \:J} + \Omega^I_{\: \:K} \wedge \Omega^K_{\: \:J}$, we write
$\Omega^I_{\: \:J} = \frac12 R^I_{\: \: JKL} \tau^K \wedge \tau^L$. 
Let $R_{V}$ denote the curvature tensor of a fiber $s^{-1}(b)$ in its induced Riemannian metric and let
$R_B$ denote the curvature tensor of $B$.

The fundamental tensors of the Riemannian
submersion are the curvature tensor $A$ and the second fundamental form $T$
\cite[Chapter 9C]{Besse (1987)}. These  satisfy the
symmetries
\begin{equation} \label{4.1}
\omega_{\alpha \beta i} \: = \: A_{\alpha \beta i} \: = \: - A_{\beta \alpha i} \: = \: - A_{\beta i \alpha} \: = \: A_{i \beta \alpha} \: = \: A_{\alpha i \beta} \: = \: - A_{i \alpha \beta}
\end{equation}
and
\begin{equation} \label{4.2}
\omega_{\alpha ij} \: = \: T_{\alpha ij} \: = \: T_{\alpha ji} \: = \: - T_{j \alpha i} \: = \: - T_{i \alpha j}.
\end{equation}

Given $f,h \in C^\infty(B)$, put $\widetilde{e}_i \: = \: e^{-f} e_i$ and $\widetilde{e}_\alpha \: = \: e^{h}  e_\alpha$. Then
$\{\widetilde{e}_I\}$ is an orthonormal basis for the metric $\widetilde{g} \: = \: e^{2f}  g^V + e^{-2h}  s^* g_B$.
Let $\widetilde{R}^I_{\: \: JKL}$ denote the components of the curvature tensor of $\widetilde{g}$, 
written in terms of the basis $\{\widetilde{e}_I\}$. One finds

\begin{align} \label{4.3}
\widetilde{A}^i_{\: \: \alpha \beta} \: = \: & e^{f+2h}  {A}^i_{\: \: \alpha \beta} \\
\widetilde{T}^i_{\: \: \alpha j} \: = \: &  e^{h} \left( {T}^i_{\: \: \alpha j} + \delta^i_{\: \: j} e_\alpha f \right)  \notag \\
\widetilde{\omega}^i_{\: \: jk} \: = \: & e^{-f} \omega^i_{\: \: jk} \notag \\
\widetilde{\omega}^i_{\: \: j\alpha} \: = \: & e^{-h} \omega^i_{\: \: j\alpha} \notag \\
\widetilde{\omega}^\alpha_{\: \: \beta \gamma} \: = \: & e^{-h} \left( \omega^\alpha_{\: \: \beta \gamma} +
\delta^\alpha_{\: \: \gamma} e_\beta h - \delta_{\beta \gamma} e_\alpha h \right) \notag
\end{align}
and
\begin{align} \label{4.4}
\widetilde{R}^i_{\: \: jkl} \: = \: & e^{-2f} R^i_{V  jkl} + e^{2h} \left[ T^i_{\: \: \alpha k} T^\alpha_{\: \: jl} -
 T^i_{\: \: \alpha l} T^\alpha_{\: \: jk} + e_\alpha f \left( \delta^i_{\: \: k} T^\alpha_{\: \: jl} - \delta_{jl} 
T^i_{\: \: \alpha k} - \delta^i_{\: \: l} T^\alpha_{\: \:jk} + \delta_{jk} T^i_{\: \: \alpha l} \right) \right. \\
& \left. - |\nabla f|^2 \left( \delta^i_{\: \: k} \delta_{jl} - \delta^i_{\: \: l} \delta_{jk} \right) \right] \notag
 \\
\widetilde{R}^i_{\: \: \alpha jk} \: = \: & e^{-f+h} \left( \nabla_j T^i_{\: \: \alpha k} - \nabla_k T^i_{\: \: \alpha j} \right)
+ e^{f+3h}  \left[ T^i_{\: \: \beta j} A^\beta_{\: \: \alpha k} - T^i_{\: \: \beta k} A^\beta_{\: \: \alpha j}
+ e_\beta f \left( \delta^i_{\: \: j} A^\beta_{\: \: \alpha k} - \delta^i_{\: \: k} A^\beta_{\: \: \alpha j} \right)
 \right] \notag  
\\
\widetilde{R}^i_{\: \: \alpha j \beta} \: = \: & e^{2h} \left[ - \nabla_\beta T^i_{\: \: \alpha j} + \nabla_j A^i_{\: \: \alpha \beta} - T^i_{\: \: \alpha k} T^k_{\: \: \beta j} \right. \notag \\
&  - (\nabla_\alpha \nabla_\beta f + e_\alpha f \: e_\beta f +  e_\alpha f \: e_\beta h +  e_\alpha h \: e_\beta f - \langle \nabla f, \nabla h \rangle \delta_{\alpha \beta})
 \delta^i_{ \: \: j} \notag \\
& \left.  - T^i_{\: \: \alpha j} e_\beta (f+h)  - T^i_{ \: \: \beta j} e_\alpha (f+h) \right]
- e^{2f+4h} A^i_{\: \: \gamma \beta}
A^\gamma_{\: \: \alpha j} 
\notag 
\\
\widetilde{R}^\alpha_{\: \: \beta i j} \: = \: & e^{2h} \left( \nabla_i A^\alpha_{\: \: \beta j} - \nabla_j A^\alpha_{\: \: \beta i} 
+ T^\alpha_{\: \: ki} T^k_{\: \: \beta j} - T^\alpha_{\: \: kj} T^k_{\: \: \beta i}  \right)
+ e^{2f+4h}  \left( A^\alpha_{\: \: \gamma i} A^\gamma_{\: \: \beta j} - A^\alpha_{\: \: \gamma j} A^\gamma_{\: \: \beta i} \right) \notag
\\
\widetilde{R}^\alpha_{\: \: \beta \gamma i} \: = \: & e^{f+3h} \left[ \nabla_\gamma A^\alpha_{\: \: \beta i}
- A^\alpha_{\: \: \beta k} T^k_{\: \: \gamma i} - A^k_{\: \: \beta \gamma} T^\alpha_{\: \: ki} + A^\alpha_{\: \: k \gamma} T^k_{\: \: \beta i} \right. \notag \\
& \left.  + 2 A^\alpha_{\: \: \beta i} e_\gamma(f+h)  + A^\alpha_{\: \: \gamma i} e_\beta(f+h) - A^\beta_{\: \: \gamma i} e_\alpha(f+h) -
A^\alpha_{\: \: \phi i} \delta_{\beta \gamma} e_\phi h
+  A^\beta_{\: \: \phi i} \delta_{\alpha \gamma} e_\phi h \right] \notag \\
\widetilde{R}^\alpha_{\: \: \beta \gamma \delta} \: = \: & e^{2h} \left[ {R}^\alpha_{B \beta \gamma \delta} - \delta^\alpha_{\: \: \delta}
(\nabla_\beta \nabla_\gamma h + e_\beta h \:  e_\gamma h) + \delta^\alpha_{\: \: \gamma}
(\nabla_\beta \nabla_\delta h + e_\beta h  \: e_\delta h) \right. \notag \\
& \left.  + \delta^\beta_{\: \: \delta}
(\nabla_\alpha \nabla_\gamma h + e_\alpha h \: e_\gamma h) - \delta^\beta_{\: \: \gamma}
(\nabla_\alpha \nabla_\delta h + e_\alpha h \: e_\delta h) - |\nabla h|^2 \left( \delta^\alpha_{\: \: \gamma} \delta_{\beta \delta} -
\delta^\alpha_{\: \: \delta} \delta_{\beta \gamma} \right)
\right] \notag \\
& + e^{2f+4h} 
\left(  2 A^\alpha_{\: \: \beta i} A^i_{\: \: \gamma \delta} - A^\alpha_{\: \: i \delta} A^i_{\: \: \beta \gamma}
+ A^\alpha_{\: \: i \gamma} A^i_{\: \: \beta \delta} \right). \notag
\end{align}
Here the covariant derivatives are with respect to the projected connection on $T^V M \oplus s^* TB$. That is,
\begin{align} \label{4.5}
\nabla_j T^i_{\: \: \alpha k} \: = \: & e_j T^i_{\: \: \alpha k} + \omega^i_{\: \: lj} T^l_{\: \: \alpha k} - \omega^l_{\: \: kj} T^i_{\: \: \alpha l}, \\
\nabla_\beta T^i_{\: \: \alpha j} \: = \: & e_\beta T^i_{\: \: \alpha j} + \omega^i_{\: \: k \beta} T^k_{\: \: \alpha j} -
\omega^\gamma_{\: \: \alpha \beta} T^i_{\: \: \gamma j} - \omega^k_{\: \: j \beta} T^i_{\: \: \alpha k}, \notag \\
\nabla_j A^i_{\: \: \alpha \beta} \: = \: & e_j A^i_{\: \: \alpha \beta} + \omega^i_{\: \: kj} A^k_{\: \: \alpha \beta}, \notag \\
\nabla_\gamma A^\alpha_{\: \: \beta i} \: = \: & e_\gamma A^\alpha_{\: \: \beta i} + \omega^\alpha_{\: \: \delta \gamma} A^\delta_{\: \: \beta i}
- \omega^\delta_{\: \: \beta \gamma} A^\alpha_{\: \: \delta i} - \omega^j_{\: \: i \gamma} A^\alpha_{\: \: \beta j}. \notag
\end{align}

In particular, if $h \: = \: 0$ then the equations simplify to
\begin{align} \label{4.6}
\widetilde{R}^i_{\: \: jkl} \: = \: & e^{-2f} R^i_{V  jkl} +  T^i_{\: \: \alpha k} T^\alpha_{\: \: jl} -
 T^i_{\: \: \alpha l} T^\alpha_{\: \: jk} + e_\alpha f \left( \delta^i_{\: \: k} T^\alpha_{\: \: jl} - \delta_{jl} 
T^i_{\: \: \alpha k} - \delta^i_{\: \: l} T^\alpha_{\: \:jk} + \delta_{jk} T^i_{\: \: \alpha l} \right) . \\
&  - |\nabla f|^2 \left( \delta^i_{\: \: k} \delta_{jl} - \delta^i_{\: \: l} \delta_{jk} \right)  \notag
 \\
 \widetilde{R}^i_{\: \: \alpha jk} \: = \: & e^{-f} \left( \nabla_j T^i_{\: \: \alpha k} - \nabla_k T^i_{\: \: \alpha j} \right) \notag \\
& + e^{f}  \left[ T^i_{\: \: \beta j} A^\beta_{\: \: \alpha k} - T^i_{\: \: \beta k} A^\beta_{\: \: \alpha j}
+ e_\beta f \left( \delta^i_{\: \: j} A^\beta_{\: \: \alpha k} - \delta^i_{\: \: k} A^\beta_{\: \: \alpha j} \right)
 \right] \notag  
\\
\widetilde{R}^i_{\: \: \alpha j \beta} \: = \: & - \nabla_\beta T^i_{\: \: \alpha j} + \nabla_j A^i_{\: \: \alpha \beta} - T^i_{\: \: \alpha k} T^k_{\: \: \beta j} 
 - (\nabla_\alpha \nabla_\beta f + e_\alpha f e_\beta f) \notag \\
&  - T^i_{\: \: \alpha j} e_\beta f  - T^i_{ \: \: \beta j} e_\alpha f 
- e^{2f} A^i_{\: \: \gamma \beta}
A^\gamma_{\: \: \alpha j} 
\notag 
\\
\widetilde{R}^\alpha_{\: \: \beta i j} \: = \: &  \nabla_i A^\alpha_{\: \: \beta j} - \nabla_j A^\alpha_{\: \: \beta i} 
+ T^\alpha_{\: \: ki} T^k_{\: \: \beta j} - T^\alpha_{\: \: kj} T^k_{\: \: \beta i} 
+ e^{2f}  \left( A^\alpha_{\: \: \gamma i} A^\gamma_{\: \: \beta j} - A^\alpha_{\: \: \gamma j} A^\gamma_{\: \: \beta i} \right) \notag
\\
\widetilde{R}^\alpha_{\: \: \beta \gamma i} \: = \: & e^{f} \left( \nabla_\gamma A^\alpha_{\: \: \beta i}
- A^\alpha_{\: \: \beta k} T^k_{\: \: \gamma i} - A^k_{\: \: \beta \gamma} T^\alpha_{\: \: ki} + A^\alpha_{\: \: k \gamma} T^k_{\: \: \beta i} \right. \notag \\
& \left.  + 2 A^\alpha_{\: \: \beta i} e_\gamma f  + A^\alpha_{\: \: \gamma i} e_\beta f - A^\beta_{\: \: \gamma i} e_\alpha f 
 \right) \notag \\
\widetilde{R}^\alpha_{\: \: \beta \gamma \delta} \: = \: &  {R}^\alpha_{B \beta \gamma \delta}
+ e^{2f} 
\left(  2 A^\alpha_{\: \: \beta i} A^i_{\: \: \gamma \delta} - A^\alpha_{\: \: i \delta} A^i_{\: \: \beta \gamma}
+ A^\alpha_{\: \: i \gamma} A^i_{\: \: \beta \delta} \right), \notag
\end{align}
If in addition the fibers of $s$ are totally geodesic and $f$ is constant then the equations further simplify to
\begin{align} \label{4.7}
\widetilde{R}^i_{\: \: jkl} \: = \: & e^{-2f} R^i_{V  jkl} 
 \\
 \widetilde{R}^i_{\: \: \alpha jk} \: = \: & 0
 \notag  
\\
\widetilde{R}^i_{\: \: \alpha j \beta} \: = \: & \nabla_j A^i_{\: \: \alpha \beta}
- e^{2f} A^i_{\: \: \gamma \beta}
A^\gamma_{\: \: \alpha j} 
\notag 
\\
\widetilde{R}^\alpha_{\: \: \beta i j} \: = \: &  \nabla_i A^\alpha_{\: \: \beta j} - \nabla_j A^\alpha_{\: \: \beta i} 
+ e^{2f}  \left( A^\alpha_{\: \: \gamma i} A^\gamma_{\: \: \beta j} - A^\alpha_{\: \: \gamma j} A^\gamma_{\: \: \beta i} \right) \notag
\\
\widetilde{R}^\alpha_{\: \: \beta \gamma i} \: = \: & e^{f} \nabla_\gamma A^\alpha_{\: \: \beta i} \notag \\
\widetilde{R}^\alpha_{\: \: \beta \gamma \delta} \: = \: &  {R}^\alpha_{B \beta \gamma \delta}
+ e^{2f} 
\left(  2 A^\alpha_{\: \: \beta i} A^i_{\: \: \gamma \delta} - A^\alpha_{\: \: i \delta} A^i_{\: \: \beta \gamma}
+ A^\alpha_{\: \: i \gamma} A^i_{\: \: \beta \delta} \right), \notag
\end{align}

\section{Bounded diameter collapse} \label{section5}

In this section we prove Theorem \ref{thm1}. As in \cite[Section 2]{Cheeger-Gromov (1986)} we shrink the Riemannian metric
in certain directions. We construct a (local) double
fibration and look at the curvature formulas that arise when simultaneously shrinking the various fibers.  We identify the terms that
could cause the curvature operator to become unbounded below.  In order to get rid of these terms, we
show that the structure group of the main fiber bundle can be reduced to actions on the fiber that are
affine in the flat directions of the fiber and isometric in the other directions. The construction of this reduction
uses the Ricci flow.

As the fF-structure is pure, after passing to a holonomy cover there are
\begin{enumerate}
\item A normal  cover $\pi : \widehat{M} \rightarrow M$, say with covering group $\Gamma$ (not necessarily finite),
\item A $T^k$-action on $\widehat{M}$,
\item A homomorphism $\rho : \Gamma \rightarrow \Aut(T^k)$ and
\item A proper submersion $s : \widehat{M} \rightarrow \widehat{B}$ which is $T^k \rtimes \Gamma$-equivariant.
\end{enumerate}
Let $Z$ denote the fiber of $s$. From the positive rank assumption,  if $\dim(Z) \: = \: 0$ then $k > 0$. The
orbit space is $\widehat{B}/(T^k \rtimes \Gamma)$.

By assumption, there is an invariant vertical metric $g^V$ with nonnegative curvature operator.
Let $g$ be an invariant metric on $M$ that extends $g^V$.  Then its lift $\widehat{g}$ to $\widehat{M}$ makes $s$ into  a 
$T^k \rtimes \Gamma$-equivariant
Riemannian submersion. 

Suppose first that $T^k$ acts freely on $\widehat{M}$.  Write $\widehat{g}$ as an orthogonal  sum
\begin{equation} \label{5.1}
\widehat{g} \: = \: \widehat{g}_1 +
\widehat{g}_2 + \widehat{g}_3,
\end{equation}
where 
\begin{enumerate}
\item $\widehat{g}_1$ is the restriction of $\widehat{g}$ to $\Ker(ds)$,
\item $\widehat{g}_2$ is the restriction of $\widehat{g}$ to the generators of the $T^k$-action and
\item $\widehat{g}_3$ is the restriction of $\widehat{g}$ to the orthogonal complement of the direct sum of $\Ker(ds)$ with the
generators of the $T^k$-action.
\end{enumerate}
Given $\epsilon > 0$, put 
\begin{equation} \label{5.2}
\widehat{g}_\epsilon \: = \:  \epsilon^2 \widehat{g}_1 +
\epsilon^2 \widehat{g}_2 +  \widehat{g}_3.
\end{equation}
Then $\widehat{g}_\epsilon$ is the pullback to $\widehat{M}$ of an
invariant Riemannian metric $g_\epsilon$ on $M$.

There is a commutative diagram
\begin{align} \label{5.3}
&\: \:  \: \:  \:  \: \: \: \: \: \: \: \: \: \: \: \: \: \widehat{M}  \\
& \:  \: \: \: \: \:  s_3 \swarrow\: \: \: \: \:  \: \: \: \: \: \searrow s  \notag \\
& \widehat{M}/T^k \: \: \: \: \: \: \: \: \: \:  \: \:  \:  \: \: \: \: \: \: \: \widehat{B} \notag \\
&  \:  \: \: \: \: \: s_2 \searrow\: \: \: \: \:  \: \: \: \: \: \swarrow s_1  \notag \\
& \: \:  \: \:  \:  \: \: \: \: \: \: \: \: \: \: \:  \widehat{B}/T^k \notag
\end{align}
of Riemannian submersions.  Let 
\begin{enumerate}
\item $\{x^\alpha\}$ be local coordinates on $\widehat{B} / T^k$, 
\item $\{x^\alpha, x^i\}$ be local coordinates on $\widehat{B}$ and 
\item $\{x^\alpha, x^I\}$ be local coordinates
on $\widehat{M} / T^k$,
\end{enumerate} so $\{x^\alpha, x^i, x^I\}$ are local coordinates on $\widehat{M}$. Let
\begin{enumerate}
\item $\{\tau^\alpha\}$ be an orthonormal local collection of $1$-forms on $\widehat{M}$ that pull back from
$\widehat{B} / T^k$, 
\item $\{\tau^\alpha, \tau^i\}$ be an orthonormal local collection of $1$-forms on
$\widehat{M}$ that pull back from $\widehat{B}$ and 
\item $\{\tau^\alpha, \tau^I\}$ be an orthonormal local collection of $1$-forms on
$\widehat{M}$ that pull back from $\widehat{M}/T^k$. 
\end{enumerate}
(Here the index $I$ has a different use than in Section \ref{section4}.)
Then $\{\tau^\alpha, \tau^i, \tau^I\}$ is a local orthonormal basis of $1$-forms on $\widehat{M}$.
We indicate the dependence of the $1$-forms on the coordinates by
\begin{enumerate}
\item $\tau^\alpha \: = \: \tau^\alpha(x^\beta)$,
\item $\tau^i \: = \: \tau^i(x^\beta, x^j)$ and
\item $\tau^I \: = \: \tau^I(x^\beta, x^J)$.
\end{enumerate}

 Let $\widetilde{R}_\epsilon$ denote the curvature tensor of
$\widehat{g}_\epsilon$.  We express components of $\widetilde{R}_\epsilon$ in terms of an
orthonormal frame for $\widehat{g}_\epsilon$. Put $f \: = \: \log \epsilon$. The components of
$\widetilde{R}_\epsilon$ that only have indices of type $\alpha$ and $i$, and at least one
index of type $i$, can be read off from (\ref{4.6}).  Similarly, the components of
$\widetilde{R}_\epsilon$ that only have indices of type $\alpha$ and $I$, and at least one
index of type $I$, can be read off from (\ref{4.6}).  The other nonzero components of
$\widetilde{R}_\epsilon$ are
\begin{align} \label{5.4}
\widetilde{R}^\alpha_{\: \: \beta \gamma \delta} \: = \: & {R}^\alpha_{B \beta \gamma \delta} + \epsilon^2 
\left( A^\alpha_{\: \: i \delta} A^i_{\: \: \beta \gamma} - 2 A^\alpha_{\: \: \beta i} A^i_{\: \: \gamma \delta} 
- A^\alpha_{\: \: i \gamma} A^i_{\: \: \beta \delta} \right)  \\
&  + \epsilon^2 \left(
A^\alpha_{\: \: I \delta} A^I_{\: \: \beta \gamma} - 2 A^\alpha_{\: \: \beta i} A^I_{\: \: \gamma \delta} 
- A^\alpha_{\: \: I \gamma} A^I_{\: \: \beta \delta} \right), \notag \\
\widetilde{R}^i_{\: \: I \alpha \beta} \: = \: & \epsilon^2 \left( A^i_{\: \: \gamma \alpha} A^\gamma_{\: \: I \beta} -
A^i_{\: \: \gamma \beta} A^\gamma_{\: \: I \alpha} \right), \notag \\
\widetilde{R}^i_{\: \: \alpha \beta I} \: = \: & \epsilon^2 A^i_{\: \: \gamma \beta} A^\gamma_{\: \: \alpha I}, \notag \\
\widetilde{R}^i_{\: \: IjJ} \: = \: & 
\epsilon^2  {A}^i_{\: \: \alpha j} {A}^\alpha_{\: \: IJ}, \notag  \\
\widetilde{R}^i_{\: \: I \alpha J} \: = \: & \epsilon A^i_{\: \: \beta \alpha} T^\beta_{\: \: IJ}, \notag \\
\widetilde{R}^i_{\: \: \alpha j I} \: = \: & \epsilon T^i_{\: \: \beta j} A^\beta_{\: \: \alpha I}. \notag
\end{align} 

As $\epsilon$ goes to zero, we see from (\ref{4.6}) that the only possibly divergent terms in the curvature operator come from the tensor $\widetilde{S}$
which has the symmetries of the curvature tensor, and whose nonzero entries are given by
\begin{align} \label{5.5}
\widetilde{S}^i_{\: \: jkl} \: = \: & \epsilon^{-2} R^i_{V jkl},\\
\widetilde{S}^i_{\: \: \alpha jk} \: = \: & \epsilon^{-1} \left( \nabla_j T^i_{\: \: \alpha k} - \nabla_k  T^i_{\: \: \alpha j} \right), \notag \\
 \widetilde{S}^I_{\: \: \alpha JK} \: = \: & \epsilon^{-1} \left( \nabla_J T^I_{\: \: \alpha K} - \nabla_K  T^I_{\: \: \alpha J} \right). \notag 
\end{align}

From (\ref{4.5}),
\begin{align} \label{5.6}
\nabla_j T^i_{\: \: \alpha k} \: = \: &
e_j T^i_{\: \: \alpha k} + \omega^i_{\: \: lj} T^i_{\: \: \alpha k} - \omega^i_{\: \: kj} T^i_{\: \: \alpha l}, \\
\nabla_J T^I_{\: \: \alpha K} \: = \: &
e_J T^I_{\: \: \alpha K} + \omega^I_{\: \: LJ} T^L_{\: \: \alpha K} - \omega^L_{\: \: KJ} T^I_{\: \: \alpha L}, \notag
\end{align}
and similarly for  $\nabla_k  T^i_{\: \: \alpha j}$ and $\nabla_K  T^I_{\: \: \alpha J}$. From the $T^k$-invariance, by
taking $\{x^i\}$ to be linear coordinates on the tori 
we can assume that $e_j T^i_{\: \: \alpha k} \: = \: 0$ and
$\omega^i_{\: \: lj} \: = \: \omega^l_{\: \: kj} \: = \: 0$. Thus $\widetilde{S}^i_{\: \: \alpha jk} \: = \: 0$.

From Section \ref{section2}, a fiber $Z$ has a finite normal cover $Z^\prime$, say with covering group $\Delta$,
which admits a $\Delta$-invariant isometric product metric 
$Z^\prime \: = \: T^a \times \prod_{m\: = \:1}^N  F_m$, where $T^a$ has a flat metric and
each $F_m$ is a compact irreducible symmetric space. 
Thinking of $Z$ as a smooth manifold,
let ${\mathcal C}$ be the space of $\Delta$-invariant structures on $Z^\prime$ consisting of a product decomposition
$Z^\prime \: = \: T^a \times \prod_{m\: = \:1}^N  F_m$, a complete affine structure on 
$T^a$ with trivial holonomy, and an irreducible symmetric space metric on each $F_m$ with volume one.
\begin{lemma} \label{newlem}
$\Diff(Z)$ acts transitively on ${\mathcal C}$.
\end{lemma}
\begin{proof}
Let $Z_1$ and $Z_2$ be manifolds diffeomorphic to $Z$, with structures ${\mathcal C}_1$ and ${\mathcal C}_2$,
respectively. We can find a $\Delta$-equivariant diffeomorphism $\phi^\prime : Z_1^\prime \rightarrow Z_2^\prime$
that preserves the product structure. From the rigidity of the affine structure on $T^a$ and the rigidity of the volume-one
symmetric space structure on $F_m$, we can assume that $\phi^\prime$ is the product of an affine map and an
isometry.
\end{proof}

Given a structure $c \in {\mathcal C}$, let $G$ be the subgroup of $\Diff(Z)$ preserving $c$.
That is, an element $\phi \in G$ has a $\Delta$-equivariant lift $\phi^\prime \in \Diff(Z^\prime)$
that acts affinely on $T^a$ and isometrically on $\prod_{m\: = \:1}^N  F_m$.
From Lemma \ref{newlem}, ${\mathcal C} \: = \: \Diff(Z)/G$. We give
${\mathcal C}$ the quotient topology.

\begin{lemma} \label{lem4} 
The fiber bundle $s : \widehat{M} \rightarrow \widehat{B}$ can be reduced to a $T^k \rtimes \Gamma$-equivariant
bundle with structure group $G$. 
\end{lemma}
\begin{proof}
Let $g_0$ be a Riemannian metric on $Z$ with nonnegative curvature operator. 
From Section \ref{section2}, a finite normal cover of $Z$ can be written as
an isometric product $Z^\prime \: = \: T^a \times \prod_{m\: = \:1}^N  F_m$.
Each factor $F_m$ is diffeomorphic to an irreducible symmetric space, but
if $F_m$ is a sphere or a complex projective space case then the metric need
not be the symmetric space metric.
We can run the Ricci flow on each $F_m$-factor, normalized so that
the volume approaches one.
From \cite{Boehm-Wilking (2008)} and \cite{Tian-Zhu (2007),Tian-Zhu (2011)}, the metric on $F_m$ converges to a symmetric
space metric with volume one.  From the proofs in these references, the limiting metric is a continuous function of
the initial metric. Running this flow fiberwise on the bundle $s : \widehat{M} \rightarrow \widehat{B}$,
in the limit each fiber obtains a structure from ${\mathcal C}$.

The bundle  $s : \widehat{M} \rightarrow \widehat{B}$ is classified by a continuous map
$f : \widehat{B} \rightarrow B\Diff(Z)$, defined up to homotopy.  With respect to the fibration
$E\Diff(Z) \rightarrow B\Diff(Z)$, put $\widehat{W} \: = \: f^* (E\Diff(Z))$, so that there is a commutative
diagram
\begin{align} \label{5.7}
& \widehat{W} \longrightarrow  E\Diff(Z) \\
& \downarrow \: \: \: \:  \: \: \: \:  \: \: \: \:  \: \: \: \:  \downarrow \notag \\
& \widehat{B} \stackrel{f}{\longrightarrow}  B\Diff(Z). \notag
\end{align}
The fibration $\widehat{W} \rightarrow \widehat{B}$ has fiber $\Diff(Z)$.
From the Ricci flow, we obtained a  continuous
$T^k \rtimes \Gamma$-invariant section of the fibration
$(\widehat{W} \times_{\Diff(Z)} {\mathcal C}) \rightarrow \widehat{B}$. Such a section
is equivalent to a  $T^k \rtimes \Gamma$-equivariant reduction of the structure group to $G$
\cite[Lemma 1.7]{Lashof (1982)}.
\end{proof}

Thus we can assume that each fiber $Z$ of $s : \widehat{M} \rightarrow \widehat{B}$ has a finite normal cover of the form
$Z^\prime \: = \: T^a \times \prod_{m\: = \:1}^N F_m$, with $F_m$ being a compact irreducible symmetric space of volume one,
and that the bundle $s : \widehat{M} \rightarrow \widehat{B}$ has a horizontal distribution whose holonomy lies in $G$.
We can find a vertical Riemannian metric $g^V$ that is compatible with this structure.
Using the same metric as before on $\widehat{B}$, along with $g^V$, we get a $T^k \rtimes \Gamma$-invariant metric $\widehat{g}$ on 
$\widehat{M}$. It follows that $T^I_{\: \:\alpha K}$ can only 
have nonzero entries when $I$ and $K$ correspond to the $T^a$-directions in the fiber. Using affine coordinates in the
affine directions, it follows that $\omega^I_{\: \: LJ}$ vanishes in such directions and $e_J T^I_{\: \:\alpha K} \: = \: 0$. Thus
$\widetilde{S}^I_{\: \: \alpha J K} \: = \: 0$. Hence
the curvature operator of $g_\epsilon$ remains bounded below as $\epsilon$ goes to zero.
The remaining claims of the theorem are straightfoward to verify; cf. \cite[p. 325-326]{Cheeger-Gromov (1986)}.

Suppose now that the $T^k$-action is not free. 
As the fF-structure is polarized, there is an integrable distribution $\widehat{\mathcal I}$
on $\widehat{M}$ so that if $U \subset M$ is an open set then the lift of $\eta({\mathcal K}(U))$ consists of the 
$\Gamma$-invariant sections of $\widehat{\mathcal I} \Big|_{\pi^{-1}(U)}$.
Decompose $\widehat{g}$ as in (\ref{5.1}),
replacing the generators of the $T^k$-action by $\widehat{\mathcal I}$. The curvature
computations are essentially local on $\widehat{M}$.  Lemma \ref{lem4} is unchanged and the rest of the
argument goes through.

\begin{remark} \label{remark1}
Instead of considering the double Riemannian submersion (\ref{5.3}), we could consider the single Riemannian submersion
$s_1 \circ s : \widehat{M} \rightarrow \widehat{B}/T^k$ with fiber $T^k \times Z$. Computing the $A$ and $T$ tensors for the Riemannian submersion
$s_1 \circ s$ would lead to the same conclusion.
\end{remark}

\section{Bounded volume collapse} \label{section6}

In this section we prove Theorem \ref{thm2}, partly along the lines of \cite[Section 3]{Cheeger-Gromov (1986)}.

Let $\{U_i\}$ be a regular atlas and let $\{F_i\}$ be a collection of invariant functions, as in Lemma \ref{lem3}. 
We can assume that the restriction of the infinitesimal fF-structure to each $U_i$ has a pure polarized subinfinitesimal
fF-structure.
Suppose first that we do not have to pass to substructures.

Let $g^V_{\inv}$ be the invariant vertical metric with positive curvature operator. Let $g_{\inv}$ be
the invariant metric on $M$ that extends $g^V_{\inv}$. Given $\epsilon > 0$, put $g_0 \: = \: (\log^2 \epsilon^{-1}) g_{\inv}$.

For $i \ge 1$, given $g_{i-1}$, we define $g_i$ inductively.  As in (\ref{5.1}), on $\widehat{U}_i$ write 
$\widehat{g}_{i-1} \: = \: \pi_i^* g_{i-1}$ as
\begin{equation} \label{6.1}
\widehat{g}_{i-1} \: = \: \widehat{g}_{i-1,1} + \widehat{g}_{i-1,2} + \widehat{g}_{i-1,3}.
\end{equation}
Put
\begin{equation} \label{6.2}
\widehat{g}_{i} \: = \: \epsilon^{2 \pi_i^* F_i} \widehat{g}_{i-1,1} +  \epsilon^{2 \pi_i^* F_i} \widehat{g}_{i-1,2} + \widehat{g}_{i-1,3}.
\end{equation}
Define $g_i$ on $U_i$ to be such that $\widehat{g}_i \: = \: \pi_i^* g_i$. Define $g_i$ on the complement of $U_i$ to be $g_{i-1}$.

Let $g_\epsilon$ be the result of performing these operators on all of the $U_i$'s.
To check that the curvature operator of $g_\epsilon$ is uniformly bounded below as $\epsilon$ goes to zero,
pick $m \in M$. Let $(T_m, F(m))$ be a maximal element with respect to the partial ordering of (\ref{3.1}). Let 
$\pi_m : \widehat{U}_m \rightarrow U_m$, $T_m$, $\Gamma_m$ and $s_m : \widehat{U}_m \rightarrow
\widehat{B}_m$ be as in Definition \ref{def3}.

We only have to check how the curvature operator at $m$ changes when going from
$g_{i-1}$ to $g_i$, for each $i$ such that $m \in U_i$, Because $F_i$ is  invariant for the entire fF-structure, 
$\pi_m^* F_i$ pulls back from a function $F^\prime_i$ on $\widehat{B}_m$.  Thus we can apply 
(\ref{4.6}), taking $f \: = \:  F^\prime_i \log(\epsilon)$.  Because of the original multiplication of $g_{\inv}$ by
$\log^2 \epsilon^{-1}$, when passing from $g_{\inv}$ to $g_0$, the terms in (\ref{4.6}) that involve
$|\nabla f|$ and $\Hess(f)$ will be uniformly bounded in $\epsilon$. We can follow the argument in 
Section \ref{section5}, up to the paragraph containing (\ref{5.6}). However, we can no longer assume that
$ \widetilde{S}^I_{\: \: \alpha JK}$ vanishes.

Suppose that the vertical curvature operator of  $\widehat{g}_{i-1}$  is uniformly bounded below by some constant
$c_{i-1}^2$. Let $D$ be a $2$-form on $\widehat{M}$. Then
\begin{align} \label{6.3}
\langle D, \widetilde{S} D \rangle \: = \: & 
D^{IJ} \widetilde{S}_{IJKL} D^{KL} + 4 D^{I\alpha} \widetilde{S}_{I \alpha JK} D^{JK} \\
\: \ge \: & \epsilon^{-2F_i(m)} c_{i-1}^2 D_{JK} D^{JK} + 4 \epsilon^{-F_i(m)} D^{I \alpha}  \left( \nabla_J T_{I \alpha K} - \nabla_K  T_{I \alpha J} \right) D^{JK} \notag \\
\: = \: & \parallel \epsilon^{-F_i(m)} c_{i-1}  D_{JK} + 2 c_{i-1}^{-1}  D^{I\alpha}   \left( \nabla_J T_{I \alpha K} - \nabla_K  T_{I \alpha J} \right) \parallel^2 \notag \\
& - 4 c_{i-1}^{-2}   \parallel  D^{I\alpha}   \left( \nabla_J T_{I \alpha K} - \nabla_K  T_{I \alpha J} \right) \parallel^2 \notag \\
\ge & - 4 c_{i-1}^{-2}   \parallel  D^{I\alpha}   \left( \nabla_J T_{I \alpha K} - \nabla_K  T_{I \alpha J} \right) \parallel^2, \notag
\end{align}
independent of $\epsilon$.

In the first step, going from $g_{\inv}$ to $g_0$, if $\epsilon < e^{-1}$ then the curvature operator cannot become more negative.  Thus it suffices to
look at the change in the curvature operator starting from $g_0$.
Because the atlas is regular, when going from $g_0$ to $g_{i-1}$, the metric on a fiber
$s_m^{-1}(\pt)$ gets multiplied by $\epsilon^{2 F_1(m) + \ldots + 2 F_{i-1}(m)}$. Then
$c_{i-1}^{-2}$ will be multiplied by the same factor. 

It follows from (\ref{6.3}) that the curvature
operator of $g_\epsilon$ stays bounded below as $\epsilon$ goes to zero.
The remaining claim of the theorem is straightforward to verify.

If instead we just have an subinfinitesimal fF-structure $\left( {\mathcal K}^\prime_i, {\mathcal D}^\prime_i \right)$ on each $U_i$, which is pure polarized, then we again
write $\widehat{g}_{i-1}$ as in (\ref{6.1}), where now 
\begin{enumerate}
\item $\widehat{g}_{i-1,1}$ is the restriction of $\widehat{g}_{i-1}$ to the
(lifted) $D^\prime_{U_i}$ directions, 
\item $\widehat{g}_{i-1,2}$  is the restriction of $\widehat{g}_{i-1}$ to the (lifted) $\eta({\mathcal K}^\prime_i)$-directions, and
 \item $\widehat{g}_{i-1,3}$ is the restriction of $\widehat{g}_{i-1}$ to the remaining orthogonal complement.
\end{enumerate}
We then define $\widehat{g}_i$ as in (\ref{6.2}). Since we assume that the curvature operator is positive as a symmetric form on 
$\Lambda^2\left( (D^\prime_{U_i})^* \right)$, the preceding argument goes through.

\section{General collapse} \label{section7}

In this section we prove Theorem \ref{thm3}, using results from \cite[Section 4]{Cheeger-Gromov (1986)}.

To establish notation, we first recall the construction of \cite[Section 4]{Cheeger-Gromov (1986)}, which applies to an
F-structure of positive rank.

Let $\epsilon > 0$ be a parameter.
Let $\{V_i\}$ be a regular atlas for the F-structure. Let $g_{\inv}$ be an invariant Riemannian
metric. Put $g_0\: = \: (\log^2 \epsilon^{-1}) g_{\inv}$.  Let $\{F_i\}$ be a 
collection of invariant functions as in Lemma \ref{lem3}. Put $\widehat{F}_i \: = \: \pi_i^* F_i$.

For $i \ge 1$, we define $g_i$ inductively from $g_{i-1}$. 
Put $\widehat{g}_{i-1,0} \: = \: \pi_i^* g_{i-1}$ on $\widehat{V}_i$.

Let $\widehat{\Sigma}_{i,j}$ be the union of the dimension-$j$ orbits of the $T_i$-action on $\widehat{V}_i$.
As in \cite[Section 4(c)]{Cheeger-Gromov (1986)}, there is a $T_i \rtimes \Gamma_i$-invariant covering
$\{\widehat{V}_{i,j}\}$ of $\widehat{V}_i$, where $\widehat{V}_{i,j}$ is a normal
neighborhood of a truncation $\widehat{\Sigma}^\prime_{i,j}$ of $\widehat{\Sigma}_{i,j}$. Let
$s_{i,j} : \widehat{V}_{i,j} \rightarrow \widehat{\Sigma}^\prime_{i,j}$ be the projection map. A certain Riemannian metric
on $\widehat{V}_{i,j}$ is constructed in \cite[Section 4(c)]{Cheeger-Gromov (1986)}, for which
$s_{i,j}$ is a $T_i \rtimes \Gamma_i$-equivariant Riemannian submersion.

Given $\widehat{p} \in \widehat{\Sigma}^\prime_{i,j}$, let $I_{i,j}(\widehat{p})$ denote the
isotropy group at $\widehat{p}$ for the $T_i$-action, with Lie algebra
${\mathfrak i}_{i,j}(\widehat{p})$. Let ${\mathfrak k}_{i,j}(\widehat{p})$ be the orthogonal complement to
${\mathfrak i}_{i,j}(\widehat{p})$, with respect to an inner product on the Lie algebra of $T_i$
defined in \cite[Section 4(c)]{Cheeger-Gromov (1986)}. Let $K_{i,j}(\widehat{p})$ be the
connected (not necessarily closed) subgroup of  $T_i$ with Lie algebra ${\mathfrak k}_{i,j}(\widehat{p})$.
Note that $K_{i,j}(\widehat{p}) \cdot \widehat{p} \: = \: T_i \cdot \widehat{p}$.
There is an isometric action of  $K_{i,j}(\widehat{p})$ on $s_{i,j}^{-1}(T_i \cdot \widehat{p})$, that extends
the action of $K_{i,j}(\widehat{p})$ on $T_i \cdot \widehat{p}$. As $j$ varies the various structures are
compatible, as described in \cite[Section 4(c)]{Cheeger-Gromov (1986)}.

Define $V_{i,j} \subset V_i$ by $V_{i,j} = \pi_i(\widehat{V}_{i.j})$. 
As in Lemma \ref{lem3}, 
we construct invariant functions
${F}_{i,j} : {V}_i \rightarrow [0,1]$ so that 
$\supp( {F}_{i,j}) \subset {V}_{i,j}$ and
${V}_i \: = \: \bigcup_j {F}_{i,j}^{-1}(1)$. Put $\widehat{F}_{i,j} \: = \: \pi_i^* F_{i,j}$.
We will define $\widehat{g}_{i-1, j}$ inductively in $j$.
Write
\begin{equation} \label{7.1}
\widehat{g}_{i-1, j-1}\: = \: \widehat{g}_{i-1, j-1}^{(1)} + \widehat{g}_{i-1, j-1}^{(2)} +   \widehat{g}_{i-1, j-1}^{(3)},
\end{equation}
where 
\begin{enumerate}
\item $\widehat{g}_{i-1, j-1}^{(1)}$ is the restriction of $\widehat{g}_{i-1, j-1}$ to the tangent space of 
each $K_{i,j}(\widehat{p})$-orbit  in $s_{i,j}^{-1}(T_i \cdot \widehat{p})$, as $\widehat{p}$ varies over
$\widehat{\Sigma}_{i,j}$, 
\item  $\widehat{g}_{i-1, j-1}^{(2)}$ is the restriction of $\widehat{g}_{i-1, j-1}$ to the orthogonal complement of
the tangent space of 
each $K_{i,j}(\widehat{p})$-orbit, in $s_{i,j}^{-1}(T_i \cdot \widehat{p})$,  as $\widehat{p}$ varies over
$\widehat{\Sigma}_{i,j}$, and
\item  $\widehat{g}_{i-1, j-1}^{(3)}$ is the restriction of $\widehat{g}_{i-1, j-1}$ to 
the orthogonal complement of the tangent  bundle  of $s_{i,j}^{-1}(T_i \cdot \widehat{p})$,  as $\widehat{p}$ varies over
$\widehat{\Sigma}_{i,j}$.
\end{enumerate}
On $\widehat{V}_{i,j}$, put
\begin{equation} \label{7.2}
\widehat{g}_{i-1, j}\: = \:  \epsilon^{2 \widehat{F}_i \widehat{F}_{i,j-1}} \widehat{g}_{i-1, j-1}^{(1)} + \widehat{g}_{i-1, j-1}^{(2)} + 
\epsilon^{-4 \widehat{F}_i \widehat{F}_{i,j}}
\widehat{g}_{i-1, j-1}^{(3)}.
\end{equation}
On the complement of $\widehat{V}_{i,j}$ in $\widehat{V}_i$, let $\widehat{g}_{i-1,j}$ be $\widehat{g}_{i-1,j-1}$ . 

Let $\widehat{g}_i$ be the result of doing this for all $j$. On $V_i$, let $g_i$ be such that
$\widehat{g}_i \: = \: \pi_i^* g_i$. On the complement of $V_i$, let $g_i$ be $g_{i-1}$. Let $g_\epsilon$ be the
result of doing this for all $i$.

We claim that as $\epsilon$ goes to zero, the curvature operator of $g_\epsilon$ is uniformly bounded below.
To see this, we can compute the curvature tensor of $\widehat{g}_{i-1,j}$ in (\ref{7.2}) using equations
(\ref{4.4}) for the locally-defined Riemannian submersion $\alpha_{i,j} : \widehat{V}_{i,j} \rightarrow \widehat{V}_{i,j}/K_{i,j}$.
The functions $\widehat{F}_i$ and $\widehat{F}_{i,j}$ pull back from functions  $\widehat{F}_i^\prime$ and
$\widehat{F}_{i,j}^\prime$ on $ \widehat{V}_{i,j}/K_{i,j}$. In applying (\ref{4.4}), we take
$f \: = \: (\log \epsilon) \widehat{F}_i^\prime \widehat{F}_{i,j}^\prime$ and
$h \: = \:  2 (\log \epsilon) \widehat{F}_i^\prime \widehat{F}_{i,j}^\prime$. Because of the expansion of
$\widehat{g}_{i-1, j-1}^{(3)}$, the terms on the right-hand side of (\ref{4.4}) with an index in such a direction will remain
bounded below. 
Thus, it is only relevant to check the curvature operator on the submanifolds
$s_{i,j}^{-1}(T_i \cdot \widehat{p})$. That the curvature operator stays bounded below on these
submanifolds follows from the proof in Section \ref{section6}, in the special case when the fibering $s_m$ is the identity map. 
As in that section, we use the fact that the functions
$F_i$ and $F_{i,j}$ are invariant for the entire F-structure. 
Note that $g_0$ had a factor of $\log^2 \epsilon^{-1}$.

It is straightforward to see that as $\epsilon$ goes to zero, $(M, g_\epsilon)$ is locally volume collapsing in the
sense of the theorem.  Note that the $T_i$-orbits in $\widehat{V}_{i,j}$ are contained in the submanifolds
$s_{i,j}^{-1}(T_i \cdot \widehat{p})$. These submanifolds are not expanded in the construction of
$g_\epsilon$, while the $K_{i,j}$-directions in them are contracted.

Now suppose that $M$ has an fF-structure of positive rank.  Let $\{U_i\}$ be a regular atlas.
Let $g_{\inv}$ be an invariant Riemannian metric.  Put $g_0 \: = \: (\log^2 \epsilon^{-1}) g_{\inv}$. 
Let $\{F_i\}$ be a collection of invariant functions as in Lemma \ref{lem3}. Put $\widehat{F}_i \: = \: 
\pi_i^* F_i$.
 
Let $s_i : \widehat{U}_i \rightarrow \widehat{V}_i$ be the fibering. (We previously denoted the base by
$\widehat{B}_i$.) We construct the covering $\{ \widehat{V}_{i,j} \}$
of $\widehat{V}_i$ as before. There is a $T_i \rtimes \Gamma_i$-action on $\widehat{V}_i$, where
$T_i$ can now be the trivial group if $s_i$ has fibers of positive dimension.
Again, there is an integrable distribution
$K_{i,j}$ on $\widehat{U}_{i,j}$.
 Put $\widehat{U}_{i,j} \: = \: s_i^{-1} \left( \widehat{V}_{i,j} \right)$.

As in the preceding discussion, we define $\widehat{g}_{i-1,j-1}$ inductively in $j$.
Write 
\begin{equation} \label{7.3}
\widehat{g}_{i-1, j-1}\: = \: \widehat{g}_{i-1, j-1}^{(0)} + \widehat{g}_{i-1, j-1}^{(1)} + \widehat{g}_{i-1, j-1}^{(2)} +   \widehat{g}_{i-1, j-1}^{(3)},
\end{equation}
where $\widehat{g}_{i-1, j-1}^{(0)}$ is the restriction of $\widehat{g}_{i-1, j-1}$ to $\Ker(ds_i)$ and  $\widehat{g}_{i-1, j-1}^{(1)}$, 
$\widehat{g}_{i-1, j-1}^{(2)}$, $\widehat{g}_{i-1, j-1}^{(3)}$ are defined as before on $\Ker(ds_i)^\perp$.
Then on $\widehat{U}_{i,j}$, we put
\begin{equation} \label{7.4}
\widehat{g}_{i-1, j}\: = \:   \epsilon^{2 \widehat{F}_i \widehat{F}_{i,j-1}} \widehat{g}_{i-1, j-1}^{(0)} + 
\epsilon^{2 \widehat{F}_i \widehat{F}_{i,j-1}} \widehat{g}_{i-1, j-1}^{(1)} + \widehat{g}_{i-1, j-1}^{(2)} + 
\epsilon^{-4 \widehat{F}_i \widehat{F}_{i,j}}
\widehat{g}_{i-1, j-1}^{(3)}.
\end{equation}
On the complement of $\widehat{U}_{i,j}$ in $\widehat{U}_i$, let $\widehat{g}_{i-1,j}$ be $\widehat{g}_{i-1,j-1}$ . 
Let $\widehat{g}_i$ be the result of doing this for all $j$. On $U_i$, let $g_i$ be such that
$\widehat{g}_i \: = \: \pi_i^* g_i$. On the complement of $U_i$, let $g_i$ be $g_{i-1}$. Let $g_\epsilon$ be the
result of doing this for all $i$.

We claim that as $\epsilon$ goes to zero, the curvature operator of $g_\epsilon$ is uniformly bounded below.
Because of the expansion in 
$\widehat{g}_{i-1, j-1}^{(3)}$, it is only relevant to check the curvature operator on the total spaces of the
fibrations $s_i$  over the submanifolds
$s_{i,j}^{-1}(T_i \cdot \widehat{p})$. That the curvature operator stays bounded below on these
total spaces follows as in Section \ref{section6}.

It is straightforward to see that as $\epsilon$ goes to zero, $(M, g_\epsilon)$ is locally volume collapsing in the
sense of the theorem.

\section{Local structure of collapsed manifolds with curvature operator bounded below} \label{section8}

In this section we give a local model, at the volume scale, for the geometry of a Riemannian manifold which is
locally volume collapsed relative to a lower bound on the curvature operator.  We describe a canonical
fF-structure on the local model. 

As an immediate application, we show that in a certain type of collapse, the curvature operator cannot stay bounded
below.

\subsection{Local model} \label{subsection8.1}

Let $M$ be a complete Riemannian manifold. We first adapt some definitions from
\cite{Kleiner-Lott (2010),Perelman (2003)}.

\begin{definition} \label{def7}
Given $p  \in M$, the curvature
scale $R_p$ at $p$ is defined as follows. If the connected component of $M$ containing $p$ has
nonnegative curvature operator then $R_p \: = \: \infty$. Otherwise, $R_p$ is the (unique) number $r > 0$
such that the smallest eigenvalue of the curvature operator on $B(p, r)$ is $- \frac{1}{r^2}$.
\end{definition}

\begin{definition} \label{def8}
Let $c_n$ denote the volume of the Euclidean unit ball in $\R^n$.
Fix $\overline{w} \in  (0, c_n)$.
Given $p \in M$, the $\overline{w}$-volume scale at $p$ is
\begin{equation} \label{8.1}
r_p \: = \: r_p(\overline{w}) \: = \: \inf \{ r>0 : \vol(B(p, r)) \: = \: \overline{w} r^n \}.
\end{equation}
If there is no such $r$ then we say that the $\overline{w}$-volume scale at $p$ is infinite.
\end{definition}

\begin{lemma} \label{lem5}
Let $(M, p)$ be a  complete pointed Riemannian manifold of dimension $n$.
\begin{enumerate}
\item
Given $\delta, \overline{w} > 0$ there is some $\delta^\prime \: = \: \delta^\prime(n, \delta, \overline{w}) > 0$ with the following property.
 Suppose that
$\vol(B(p, R_p)) \le \delta^\prime  R_p^n$. Then $r_p \le \delta R_p$.
\item Given $r, \delta, \overline{w} > 0$ there is some $\delta^\prime \: = \: \delta^\prime(n, r, \delta, \overline{w}) > 0$ with the following property.
Suppose that the curvature operator of $M$ is bounded below by $- \Id$. Suppose that
$\vol(B(p, r)) \le \delta^\prime$. Then $r_p \le \delta R_p$.
\end{enumerate}
\end{lemma}
\begin{proof}
\begin{enumerate}
\item
Without loss of generality, we can assume that $\delta < 1$. Put $\delta^\prime \: = \: \frac12 \overline{w} \delta^n$.
Since $\vol(B(p, R_p)) \le \delta^\prime  R_p^n < \frac12 \overline{w} R_p^n$,
the definition of $r_p$ implies that $r_p < R_p$. Then
\begin{equation} \label{8.2}
\overline{w} r_p^n \: = \: 
\vol \left( B(p, r_p) \right) \le \vol \left( B(p, R_p ) \right) \le \delta^\prime R_p^n \: = \:  \frac12 \overline{w} \delta^n  R_p^n,
\end{equation}
from which the claim follows.
\item 
From the definition of $R_p$, we have $R_p \ge 1$.
Put $\delta^\prime \: = \: \frac12 \overline{w} \min(r^n, \delta^n)$. Since,
$\vol(B(p,r)) \le \delta^\prime \le \frac12 \overline{w} r^n$, the definition of $r_p$ implies $r_p < r$. Then
\begin{equation} \label{8.3}
\overline{w} r_p^n \: = \: 
\vol \left( B(p, r_p) \right) \le \vol \left( B(p, r ) \right) \le \delta^\prime \le  \frac12 \overline{w} \delta^n,
\end{equation}
from which the claim follows.
\end{enumerate}
\end{proof}

We define pointed $C^K$-nearness between two complete pointed Riemannian manifolds $(M_1, p_1)$ and
$(M_2, p_2)$ as usual  in terms of the existence of a basepoint-preserving
$C^{K+1}$-smooth map $F : M_1 \rightarrow M_2$ so that $F^* g_2$ is $C^K$-close to
$g_1$ on large balls around $p_1$.

The next proposition gives a local model for the geometry of a manifold which is locally volume collapsed with respect to
a lower bound on the curvature operator, under an additional smoothing assumption at the volume scale.

\begin{proposition} \label{prop1}
Given $n \in \Z^+$,  an integer $K \ge 10$, a function $A : (0, \infty) \times (0, \infty) \rightarrow (0, \infty)$ and $\epsilon > 0$, there are some 
$\overline{w}, \delta  > 0$ with the following property.
Let $(M, p)$ be a pointed complete Riemannian manifold of dimension $n$.  Let $r_p$ denote the $\overline{w}$-volume scale at $p$.
Suppose that 
\begin{enumerate}
\item $r_p \le \delta R_p$ and
\item For all $k \in [0,K]$ and $C < \delta^{-1}$, we have $| \nabla^k \Riem | \le A(C, \overline{w}) r_p^{-(k+2)}$
on $B(p, Cr_p)$.
\end{enumerate}
Then
$\left( \frac{1}{r_p} M, p \right)$ is $\epsilon$-close in the pointed $C^K$-topology to a pointed complete Riemannian manifold
$({\mathcal M}, p^\prime)$ for which ${\mathcal M}$ is an isometric quotient
$(\widetilde{S} \times Y)/\Gamma$, where
\begin{enumerate}
\item $S$ is a compact Riemannian manifold having nonnegative curvature operator, with fundamental group $\Gamma$ and
 $\widetilde{S}$ as its universal cover, 
\item $(Y, y)$ is a pointed complete Riemannian manifold with nonnegative curvature operator that is diffeomorphic to $\R^k$, with $k \in [0,n]$,
\item $\Gamma$ acts by isometries on $Y$, fixing $y$, and
\item The Tits cone of $Y$ is $\epsilon$-close in the pointed Gromov-Hausdorff topology to a complete pointed
Alexandrov space of dimension less than $n$.
\end{enumerate}
\end{proposition}
\begin{proof}
 Fix $\overline{w} > 0$ for the moment. Leaving off conclusion (4) for the moment,  suppose that the proposition is not true.
Then there is a positive sequence $\{\delta_i \}_{i\: = \:1}^\infty$ converging to zero
and a sequence of pointed Riemannian manifolds $\{(M_i, p_i)\}_{i\: = \:1}^\infty$
so that $(M_i, p_i)$
satisfies the hypotheses of the proposition, with $\delta \: = \: \delta_i$, but does not satisfy the conclusion of the proposition.
Using assumption (2) of the proposition, after passing to a subsequence we can assume that $\lim_{i \rightarrow \infty} 
\left( \frac{1}{r_{p_i}} M_i, p_i \right)
\: = \: ({\mathcal M}, p^\prime )$ exists in the pointed $C^K$-topology. From the definition of $R_{p_i}$, assumption (1)
of the proposition implies that ${\mathcal M}$ has nonnegative curvature operator.  Let $S$ be the soul of ${\mathcal M}$.
Then conclusions (1)-(3) of the proposition follow from the result of \cite{Noronha (1989)}, as mentioned in Section \ref{section2}.

Finally, if $\dim(Y) < n$ then the Tits cone of $Y$ is already  of dimension less than $n$.  Suppose that
$\dim(Y) \: = \: n$, so  ${\mathcal M} \: = \: Y$. Now $\vol(B(p^\prime, 1))  \: = \:  \overline{w}$. The Bishop-Gromov inequality implies that for all $R \ge 1$,
we have $\vol(B(p^\prime, R)) \le \overline{w} R^n$. Taking $\overline{w}$ sufficiently small, we can ensure
that the Tits cone of ${\mathcal M}$ has pointed Gromov-Hausdorff distance less than $\epsilon$ from some 
complete pointed Alexandrov space of dimension less than $n$.
\end{proof}

\begin{corollary} \label{cor2}
Given $n \in \Z^+$,  an integer $K \ge 10$, a function $A : (0, \infty) \times (0, \infty) \rightarrow (0, \infty)$ and $r, \epsilon > 0$, there are some 
$\overline{w}, \delta^\prime  > 0$ with the following property.
Let $M$ be a complete Riemannian manifold of dimension $n$ with curvature operator bounded below by $- \Id$.  
Let $r_p$ denote the $\overline{w}$-volume scale at $p$.
Suppose that for some $p \in M$,
\begin{enumerate}
\item $\vol(B(p,r)) \le \delta^\prime$ and
\item For all $k \in [0,K]$ and $C < \delta^{-1}$, we have $| \nabla^k \Riem | \le A(C, \overline{w}) r_p^{-(k+2)}$
on $B(p, Cr_p)$.
\end{enumerate}
Then the conclusion of Proposition \ref{prop1} holds at $p$.
\end{corollary}
\begin{proof}
This follows from Lemma \ref{lem5}.(2) and Proposition \ref{prop1}.
\end{proof}

\subsection{Canonical fF-structure} \label{subsection8.2}

We describe a canonical fF-structure on ${\mathcal M} = (\widetilde{S} \times Y)/\Gamma$; cf. \cite[Section 2]{Cheeger-Gromov (1990)}.
From Section \ref{section2}, there is a finite cover $S^\prime$ of $S$ which is isometric
to the product of a torus $T^l$ with a simply-connected compact manifold $N^\prime$, the latter of which being
the isometric product of irreducible factors each of which is 
\begin{enumerate}
\item A Riemannian manifold that is diffeomorphic to a sphere,
\item A K\"ahler manifold that is biholomorphic to a complex projective space or
\item A symmetric space with nonnegative sectional curvature.
\end{enumerate}

Let ${\mathcal M}^\prime$ be the pullback of ${\mathcal M}$ to $S^\prime$, so that the diagram
\begin{align} \label{8.4}
& {\mathcal M}^\prime \rightarrow {\mathcal M} \\
& \downarrow \: \: \: \: \: \: \: \: \: \: \: \:  \downarrow \notag \\
& S^\prime \: \: \rightarrow \: \:  S \notag
\end{align}
commutes. The $N^\prime$ factor splits off isometrically from the bundle ${\mathcal M}^\prime \rightarrow S^\prime$, so ${\mathcal M}^\prime$
is a product
${\mathcal M}^\prime \: = \: N^\prime \times Z^\prime$. The remaining bundle $Z^\prime \rightarrow T^l$ has a local isometric product structure, wtih
fiber $Y$. There is a corresponding holonomy representation
$\rho : \Z^l \rightarrow \Isom(Y, y)$. 

As the differential $D_y : \Isom(Y, y) \rightarrow O(T_yY)$ is injective,
we can equally well consider the homomorphism $D_y \circ \rho : \Z^l \rightarrow O(T_yY)$. Put
$C \: = \: \overline{\Image(D_y \circ \rho)}$. Then $C$ is a compact Abelian Lie group, and so fits into a split
exact sequence $1 \rightarrow T^i \rightarrow C \rightarrow F \rightarrow 1$ for some finite abelian 
group $F$.  In particular, there is a surjective homomorphism
\begin{equation} \label{8.5}
\eta : \Z^l \stackrel{D_y \circ \rho}{\longrightarrow} C \longrightarrow F.
\end{equation}
Put
$\widehat{T}^l \: = \: \R^l/\Ker(\eta)$. Let $Z^{\prime \prime}$ be the pullback of
$Z^\prime$ to $\widehat{T}^l$, so that the diagram
\begin{align} \label{8.6}
& Z^{\prime \prime} \rightarrow \: \:  Z^\prime \\
& \downarrow \: \: \: \: \: \: \: \: \: \: \: \:  \downarrow \notag \\
& \widehat{T}^l \: \: \rightarrow \: \:  T^l \notag
\end{align}
commutes. Put ${\mathcal M}^{\prime \prime} \: = \: N^\prime \times Z^{\prime \prime}$. It is a finite cover of ${\mathcal M}$.

We claim that there is an effective isometric action on $Z^{\prime \prime}$ by a torus whose dimension is at least $l$.
Since the universal cover $\widetilde{Z^{\prime \prime}}$ splits isometrically as $\R^l \times Y$, there is an evident
$\R^l \times T^i$-action on  $\widetilde{Z^{\prime \prime}}$. There is an inclusion $\Ker(\eta) \rightarrow 
\R^l \times T^i$, which is the product of the inclusion $\Ker(\eta) \rightarrow \R^l$ and the
restriction of $D_y \circ \rho$ to $\Ker(\eta)$. The action of $\R^l \times T^i$ on $Z^{\prime \prime} \: = \:
(\R^l \times Y)/\Ker(\eta)$ factors through an effective action of the
compact connected Abelian Lie group $(\R^l \times T^i)/\Ker(\eta)$. 

This gives an pure fF-structure on ${\mathcal M}$, since the
product bundle ${\mathcal M}^{\prime \prime} \rightarrow Z^{\prime \prime}$ is $(\R^l \times T^i)/\Ker(\eta)$-equivariant. The fF-structure
has positive rank unless $S$ is a point, in which case ${\mathcal M}$ is
diffeomorphic to $\R^n$.

\begin{remark} \label{remark2}
Regarding a possible converse to Theorem \ref{thm3}, Corollary \ref{cor2} gives a model for the local geometry of $M$, under a smoothing
assumption, and this subsection gives an fF-structure on the local model.  This fF-structure has positive rank unless the local model is
diffeomorphic to $\R^n$. We note that the latter possibility does not occur in the analogous discussion for double sided
curvature bounds
\cite[Section 2]{Cheeger-Gromov (1990)}. In that case, the local model is flat; if it were diffeomorphic to $\R^n$ then 
conclusion (4) of Proposition \ref{prop1} would be violated.  In contrast, in our case a model space diffeomorphic to $\R^n$,
which is collapsed at infinity, could well occur.

After possibly removing some disjoint topological  balls, centered around points whose local models are diffeomorphic to $\R^n$,
it is plausible that the local fF-structures on the complement  can be glued together to form
a global fF-structure of positive rank, for which the fibers of the submersions are manifolds that admit  Riemannian metrics
of nonnegative curvature operator.  

One approach to proving this involves showing that if $(M, g)$
is locally collapsed, with respect to a lower bound on the
curvature operator, then the local geometry around a point $p \in M$ is
appropriately close, at the volume scale, to one in which the smoothing assumption (2) of Corollary \ref{cor2} is satisfied.
This is true in the bounded curvature case \cite[Pf. of Theorem 2.3]{Cheeger-Tian (2006)}.
Related smoothing results appear in \cite[Section 4]{Cabezas-Rivas-Wilking (2011)} and
\cite[Theorem 2.2]{Simon (2009)}.

If the conclusion of Proposition \ref{prop1} is satisfied at each point
of $M$ then 
a second issue is to glue together the local fF-structures to a global fF-structure. Related gluing results  are in
\cite[Section 5]{Cheeger-Gromov (1990)} (for the group actions) and \cite{Kleiner-Lott (2010)}
(for the fiberings).
  
\end{remark}

\begin{remark} \label{remark3}
In the bounded curvature case, there are analogs of the F-structure versions of Theorems \ref{thm1}-\ref{thm3} when the
F-structure is replaced by a Nil-structure.  This was stated in \cite[Remark 2.1]{Cheeger-Gromov (1986)} and proved in
\cite{Cai-Rong (2009)}. We expect that Theorems \ref{thm1}-\ref{thm3} can be extended to the setting of fNil-structures. 
 In the other direction, the
main theorem of \cite{Cheeger-Fukaya-Gromov (1992)} describes the local geometry
of a locally collapsed manifold with a prescribed double-sided curvature bound, in {\em all} of the collapsed directions, in terms of a Nil-structure.
To have a similar result with curvature operator bounded below, the first
step would be to characterize the manifolds with almost nonnegative curvature operator (ANCO),
which is an interesting question in its own right.

It is clear that almost flat manifolds are ANCO, compact manifolds with
nonnegative curvature operator are ANCO, and products of ANCO manifolds are ANCO. Using (\ref{4.7}), one sees that 
if $M$ is ANCO then a principal $T^k$-bundle over $M$ is ANCO.  As noted in
\cite{Herrmann-Sebastian-Tuschmann (2012)}, this already shows that there are simply-connected
ANCO manifolds that do not admit a Riemannian metric of nonnegative curvature operator, coming from circle bundles over
products of complex projective spaces. See \cite{Herrmann-Sebastian-Tuschmann (2012)} for further discussion.
\end{remark}

\subsection{A special type of collapse} \label{subsection8.3}

Let $M$ be a compact Riemannian manifold.  Let $G$ be a compact connected Lie group of positive dimension that acts effectively and isometrically on $M$.
Given $m \in M$, let $G_m$ denote the isotropy subgroup at $m$. By the slice theorem, there are a finite-dimensional vector space 
$V_m$ and a homomorphism $\rho_m : G_m \rightarrow \GL(V_m)$ 
so that there is a $G$-invariant neighborhood of the orbit $G \cdot m$ which is $G$-diffeomorphic to $G \times_{G_m} V_m$.

\begin{proposition} \label{prop2}
Give $G$ a bi-invariant Riemannian metric.
Suppose that as $\epsilon \rightarrow 0$, the curvature operator of $M_\epsilon \: = \: M \times_G \epsilon G$ stays uniformly bounded below.
Then for each $m \in M$, the orbit $G/G_m$ is a locally symmetric space. Let ${\mathfrak g} \: = \: {\mathfrak g}_m + {\mathfrak m}_m$ be the
orthogonal decomposition, with $[{\mathfrak g}_m, {\mathfrak m}_m] \subset {\mathfrak m}_m$.  Then in addition, for all $X, Y \in {\mathfrak m}_m$, we have
$\rho_m \left( [X, Y]_{{\mathfrak g}_m} \right) \: = \: 0$.
\end{proposition}
\begin{proof}
As $\epsilon$ goes to zero, the spaces $M_\epsilon$ converge in the Gromov-Hausdorff topology to the lower-dimensional space
$M/G$.
The rescaled pointed spaces $\left( \frac{1}{\epsilon} M_\epsilon, m \right)$ converge in the pointed Gromov-Hausdorff topology to $G \times_{G_m} V_m$.
Hence the $\overline{w}$-volume scale at $m$ is proportionate to $\epsilon$. It follows from Proposition \ref{prop1} that
$G/G_m$ is a locally symmetric space and $G \times_{G_m} V_m$ is a vector bundle over $G/G_m$ with a local isometric product
structure. From \cite[Theorem 11.1]{Kobayashi-Nomizu (1963)}, the curvature of the principal bundle $G \rightarrow G/G_m$ is given by
$\Omega(X,Y) \: = \: - \frac12  [X, Y]_{{\mathfrak g}_m}$ for $X, Y \in {\mathfrak m}_m$. As $G \times_{G_m} V_m$ is the associated vector bundle, if its
connection is flat then
$\rho_m \left( [X, Y]_{{\mathfrak g}_m} \right) \: = \: 0$.
\end{proof}

\bibliographystyle{acm}

\end{document}